\documentclass[a4paper,draft]{amsart}
\usepackage{latexsym}\usepackage{ifthen}\usepackage[leqno]{amsmath}\usepackage{enumerate}
\usepackage{hyphenat}\usepackage{amstext,amsbsy,amsopn,amsthm,amsgen,amsfonts,amscd,amsxtra,upref}
\usepackage{mathrsfs}\usepackage{euscript}\usepackage{amssymb}

\swapnumbers
\theoremstyle{plain}\newtheorem{Thm}{Theorem}[section]
\newtheorem{Lem}[Thm]{Lemma}\newtheorem{Cor}[Thm]{Corollary}\newtheorem{Pro}[Thm]{Proposition}
\theoremstyle{definition}
\newtheorem{Def}[Thm]{Definition}\newtheorem{Exm}[Thm]{Example}\newtheorem{Prb}[Thm]{Problem}
\theoremstyle{remark}

\newcommand{\myEmail}{piotr.niemiec@uj.edu.pl}
\newcommand{\myAddress}[1]{\noindent{}\ITE{\equal{#1}{}}{}{Piotr Niemiec\\{}}
   In\-sty\-tut Ma\-te\-ma\-ty\-ki\\{}Wy\-dzia\l{} Ma\-te\-ma\-ty\-ki i In\-for\-ma\-ty\-ki\\{}
   U\-ni\-wer\-sy\-tet Ja\-giel\-lo\'{n}\-ski\\{}ul. \L{}o\-ja\-sie\-wi\-cza 6\\{}
   30-348 Kra\-k\'{o}w\\{}Poland}
\newcommand{\myData}[1][Piotr Niemiec]{\author[P. Niemiec]{Piotr Niemiec}\address{\myAddress{#1}}
   \email{\myEmail}}

\newcommand{\CCC}{\mathbb{C}}\newcommand{\RRR}{\mathbb{R}}\newcommand{\TTT}{\mathbb{T}}
\newcommand{\AAa}{\CMcal{A}}\newcommand{\BBb}{\CMcal{B}}\newcommand{\CCc}{\CMcal{C}}
\newcommand{\GGg}{\CMcal{G}}\newcommand{\HHh}{\CMcal{H}}\newcommand{\IIi}{\CMcal{I}}
\newcommand{\KKk}{\CMcal{K}}\newcommand{\MMm}{\CMcal{M}}\newcommand{\PPp}{\CMcal{P}}
\newcommand{\RRr}{\CMcal{R}}\newcommand{\UUu}{\CMcal{U}}
\newcommand{\EeE}{\EuScript{E}}\newcommand{\FfF}{\EuScript{F}}\newcommand{\IiI}{\EuScript{I}}
\newcommand{\MmM}{\EuScript{M}}
\newcommand{\Bb}{\mathfrak{B}}\newcommand{\Gg}{\mathfrak{G}}\newcommand{\Mm}{\mathfrak{M}}
\newcommand{\Uu}{\mathfrak{U}}\newcommand{\Xx}{\mathfrak{X}}\newcommand{\jJ}{\mathfrak{j}}
\newcommand{\uU}{\mathfrak{u}}\newcommand{\vV}{\mathfrak{v}}
\newcommand{\tTT}{\pmb{T}}\newcommand{\xXx}{\pmb{x}}

\newcommand{\SECT}[1]{\section{#1}\renewcommand{\theequation}{\arabic{section}-\arabic{equation}}
   \setcounter{equation}{0}}
\newcommand{\ITE}[3]{\ifthenelse{#1}{#2}{#3}}\newcommand{\ITEE}[3]{\ITE{\equal{#1}{#2}}{#3}{}}

\newcommand{\tr}{\operatorname{tr}}\newcommand{\scalarr}{\langle\cdot,\mathrm{-}\rangle}
\newcommand{\scalar}[2]{\langle #1,#2\rangle}\newcommand{\leqsl}{\leqslant}
\newcommand{\geqsl}{\geqslant}\newcommand{\epsi}{\varepsilon}\newcommand{\varempty}{\varnothing}
\newcommand{\dd}{\colon}\newcommand{\dint}[1]{\,\textup{d} #1}

\newcommand{\iaoi}{if and only if}\newcommand{\tfcae}{the following conditions are equivalent:}

\newcommand{\DEF}[1]{Definition~\textup{\ref{def:#1}}}
\newcommand{\EXM}[1]{Example~\textup{\ref{exm:#1}}}
\newcommand{\LEM}[1]{Lemma~\textup{\ref{lem:#1}}}
\newcommand{\PRO}[1]{Proposition~\textup{\ref{pro:#1}}}
\newcommand{\THM}[1]{Theorem~\textup{\ref{thm:#1}}}

\newenvironment{dfn}[1]{\begin{Def}\label{def:#1}}{\end{Def}}
\newenvironment{exm}[1]{\begin{Exm}\label{exm:#1}}{\end{Exm}}
\newenvironment{lem}[1]{\begin{Lem}\label{lem:#1}}{\end{Lem}}
\newenvironment{prb}[1]{\begin{Prb}\label{prb:#1}}{\end{Prb}}
\newenvironment{pro}[1]{\begin{Pro}\label{pro:#1}}{\end{Pro}}
\newenvironment{thm}[1]{\begin{Thm}\label{thm:#1}}{\end{Thm}}

\newcommand{\bibITEM}[2]{\ITE{\equal{#2}{}}{\bibitem{#1} }{\bibitem[#2]{#1} }}
\newcommand{\BIB}[8]{
   \bibITEM{#1}{#8} #2, \textit{#3}, #4{} \textbf{#5} (#6), #7.}
\newcommand{\myBIB}[7][P. Niemiec]{\ITE{\equal{#7}{*;}\or\equal{#7}{*+}}{}{#1, \textit{#2}, }
   #3{}\ITE{\equal{#4}{}}{}{ \textbf{#4}} (#5), #6\ITE{\equal{#7}{*;}\or\equal{#7}{*+}}{}{.}}
\newcommand{\BIb}[6]{
   \bibITEM{#1}{#6} #2, \textit{#3}, #4, #5.}
\newcommand{\BiB}[9]{
   \bibITEM{#1}{#9} #2, \textit{#3}, #4{} \textit{#5}, #6, #7, #8.}

\newcommand{\jRN}[2][]{
   \ITEE{#2}{ActaM}{\ITE{\equal{#1}{+}}
      {Acta Mathematica}{Acta Math.}}
   \ITEE{#2}{ACS}{\ITE{\equal{#1}{+}}
      {Applied Categorical Structures}{Appl. Categ. Structures}}
   \ITEE{#2}{AnnM}{\ITE{\equal{#1}{+}}
      {Annals of Mathematics}{Ann. Math.}}
   \ITEE{#2}{DissM}{\ITE{\equal{#1}{+}}
      {Dissertationes Mathematicae (Roz\-pra\-wy Ma\-te\-ma\-tycz\-ne)}
      {Dissertationes Math. (Roz\-pra\-wy Mat.)}}
   \ITEE{#2}{ExtrM}{\ITE{\equal{#1}{+}}
      {Extracta Mathematicae}{Extracta Math.}}
   \ITEE{#2}{IllinoisJM}{\ITE{\equal{#1}{+}}
      {Illinois Journal of Mathematics}{Illinois J. Math.}}
   \ITEE{#2}{IndianaUMJ}{\ITE{\equal{#1}{+}}
      {Indiana University Mathematical Journal}{Indiana Univ. Math. J.}}
   \ITEE{#2}{InvM}{\ITE{\equal{#1}{+}}
      {Inventiones Mathematicae}{Invent. Math.}}
   \ITEE{#2}{JAT}{\ITE{\equal{#1}{+}}
      {Journal of Approximation Theory}{J. Approx. Theory}}
   \ITEE{#2}{JKoreanMS}{\ITE{\equal{#1}{+}}
      {Journal of the Korean Mathematical Society}{J. Korean Math. Soc.}}
   \ITEE{#2}{MMag}{\ITE{\equal{#1}{+}}
      {Mathematics Magazine}{Math. Mag.}}
   \ITEE{#2}{MAMS}{\ITE{\equal{#1}{+}}
      {Memoirs of the American Mathematical Society}{Mem. Amer. Math. Soc.}}
   \ITEE{#2}{MZ}{\ITE{\equal{#1}{+}}
      {Math. Z.}{Math. Z.}}
   \ITEE{#2}{NAMS}{\ITE{\equal{#1}{+}}
      {Notices of the American Mathematical Society}{Notices Amer. Math. Soc.}}
   \ITEE{#2}{PacJM}{\ITE{\equal{#1}{+}}
      {Pacific Journal of Mathematics}{Pacific J. Math.}}
   \ITEE{#2}{PAMS}{\ITE{\equal{#1}{+}}
      {Proceedings of the American Mathematical Society}{Proc. Amer. Math. Soc.}}
   \ITEE{#2}{TAMS}{\ITE{\equal{#1}{+}}
      {Transactions of the American Mathematical Society}{Trans. Amer. Math. Soc.}}
   }

\newcommand{\paplist}[3][]{
   \ITEE{#3}{RBhatia1997}{
      \BIb{#2}{R. Bhatia}
         {Matrix Analysis}
         {Springer, New York}{1997}{#1}}
   \ITEE{#3}{EBishop1961}{
      \BIB{#2}{E. Bishop}
         {A generalization of the Stone\hyp{}Weierstrass theorem}
         {\jRN{PacJM}}{11}{1961}{777--783}{#1}}
   \ITEE{#3}{BBlackadar2006}{\BIb{#2}{B. Blackadar}{Operator Algebras. 
         Theory of $C^*$\hyp{}algebras and von Neumann algebras \textup{(Encyclopaedia 
         of Mathematical Sciences, vol. 122: Operator Algebras and Non\hyp{}Commutative Geometry 
         III)}}{Springer\hyp{}Verlag, Berlin\hyp{}Heidelberg}{2006}{#1}}
   \ITEE{#3}{WMChing1974}{
      \BIB{#2}{W.-M. Ching}
         {Topologies on the quasi-spectrum of a $C^*$\hyp{}algebra}
         {\jRN{PAMS}}{46}{1974}{273--276}{#1}}
   \ITEE{#3}{JDixmier1977}{
      \BIb{#2}{J. Dixmier}
         {$C^*$\hyp{}algebras}
         {North-Holland Publ. Co., Amsterdam}{1977}{#1}}
   \ITEE{#3}{JErnest1976}{
      \BIB{#2}{J. Ernest}
         {Charting the operator terrain}
         {\jRN{MAMS}}{171}{1976}{207 pp}{#1}}
   \ITEE{#3}{JMGFell1960a}{
      \BIB{#2}{J.M.G. Fell}
         {$C^*$-algebras with smooth dual}
         {\jRN{IllinoisJM}}{4}{1960}{221--230}{#1}}
   \ITEE{#3}{JMGFell1960b}{
      \BIB{#2}{J.M.G. Fell}
         {The dual spaces of $C^*$-algebras}
         {\jRN{TAMS}}{94}{1960}{365--403}{#1}}
   \ITEE{#3}{JMGFell1961}{
      \BIB{#2}{J.M.G. Fell}
         {The structure of algebras of operator fields}
         {\jRN{ActaM}}{106}{1961}{233--280}{#1}}
   \ITEE{#3}{MIGarrido,FMontalvo1991}{
      \BIB{#2}{M.I. Garrido and F. Montalvo}
         {On some generalizations of the Kakutani\hyp{}Stone and Stone\hyp{}Weierstrass theorems}
         {\jRN{ExtrM}}{6}{1991}{156--159}{#1}}
   \ITEE{#3}{JGlimm1960}{
      \BIB{#2}{J. Glimm}
         {A Stone\hyp{}Weierstrass theorem for $C^*$\hyp{}algebras}
         {\jRN{AnnM}}{72}{1960}{216--244}{#1}}
   \ITEE{#3}{DWHadwin1976}{
      \BIB{#2}{D.W. Hadwin}
         {An operator\hyp{}valued spectrum}
         {\jRN{NAMS}}{23}{1976}{A-163}{#1}}
   \ITEE{#3}{DWHadwin1977}{
      \BIB{#2}{D.W. Hadwin}
         {An operator\hyp{}valued spectrum}
         {\jRN{IndianaUMJ}}{26}{1977}{329--340}{#1}}
   \ITEE{#3}{DHofmann2002}{
      \BIB{#2}{D. Hofmann}
         {On a generalization of the Stone\hyp{}Weierstrass theorem}
         {\jRN{ACS}}{10}{2002}{569--592}{#1}}
  \ITEE{#3}{JSKim,ChRKim,SGLee1980}{
      \BIB{#2}{J.S. Kim, Ch.R. Kim, S.G. Lee}
         {Reducing operator valued spectra of a Hilbert space operator}
         {\jRN{JKoreanMS}}{17}{1980}{123--129}{#1}}
   \ITEE{#3}{SGLee1980}{
      \BIB{#2}{S.G. Lee}
         {Remarks on reducing operator valued spectrum}
         {\jRN{JKoreanMS}}{16}{1980}{131--136}{#1}}
   \ITEE{#3}{RLongo1984}{
      \BIB{#2}{R. Longo}
         {Solution of the factorial Stone-Weierstrass conjecture. An application of the theory 
         of standard split $W^*$-inclusions}
         {\jRN{InvM}}{76}{1984}{145--155}{#1}}
   \ITEE{#3}{KLowner1934}{
      \BIB{#2}{K. L\"{o}wner}
         {\"{U}ber monotone Matrixfunctionen}
         {\jRN{MZ}}{38}{1934}{177--216}{#1}}
   \ITEE{#3}{pn2012b}{\bibITEM{#2}{#1} \mypaplist{pn19}}
   \ITEE{#3}{CPearcy,NSalinas1974}{
      \BIB{#2}{C. Pearcy and N. Salinas}
         {Finite-dimensional representations of separable $C^*$-algebras}
         {\jRN{NAMS}}{21}{1974}{A-376}{#1}}
   \ITEE{#3}{SPopa1984}{
      \BIB{#2}{S. Popa}
         {Semiregular maximal abelian $*$-subalgebras and the solution to the factor state 
         Stone-Weierstrass problem}
         {\jRN{InvM}}{76}{1984}{157--161}{#1}}
   \ITEE{#3}{TSaito1972}{
      \BiB{#2}{T. Sait\^{o}}{Generations of von Neumann algebras}
         {Lecture Notes in Math. vol. 247}{\textup{(}Lecture on Operator Algebras\textup{)}}
         {Springer, Berlin\hyp{}Heidelberg\hyp{}New York}{1972}{435--531}{#1}}
   \ITEE{#3}{SSakai1971}{
      \BIb{#2}{S. Sakai}
         {$C^*$\hyp{}Algebras and $W^*$\hyp{}Algebras}
         {Springer\hyp{}Verlag, Berlin\hyp{}Heidelberg\hyp{}New York}{1971}{#1}}
   \ITEE{#3}{MHStone1937}{
      \BIB{#2}{M.H. Stone}
         {Application of the theory of Boolean rings to general topology}
         {\jRN{TAMS}}{41}{1937}{375--481}{#1}}
   \ITEE{#3}{MHStone1948}{
      \BIB{#2}{M.H. Stone}
         {The generalized Weierstrass approximation theorem}
         {\jRN{MMag}}{21}{1948}{167--184}{#1}}
   \ITEE{#3}{MTakesaki2002}{
      \BIb{#2}{M. Takesaki}{Theory 
         of Operator Algebras I \textup{(Encyclopaedia of Mathematical Sciences, Volume 124)}}
         {Springer\hyp{}Verlag, Berlin\hyp{}Heidelberg\hyp{}New York}{2002}{#1}}
   \ITEE{#3}{VTimofte2005}{
      \BIB{#2}{V. Timofte}
         {Stone\hyp{}Weierstrass theorems revisited}
         {\jRN{JAT}}{136}{2005}{45--59}{#1}}
   }

\newcommand{\mypaplist}[2][]{
   \ITEE{#2}{pn19}{
      \myBIB{Unitary equivalence and decompositions of finite systems of closed densely defined 
         operators in Hilbert spaces}{\jRN{DissM}}{482}{2012}{1--106}{#1}}
   }

\begin{document}

\title{Elementary approach to homogeneous $C^*$-algebras}
\myData\thanks{The author gratefully acknowledges the assistance of the Polish Ministry 
   of Sciences and Higher Education grant NN201~546438 for the years 2010--2013.}
\begin{abstract}
A $C^*$-algebra is \textit{$n$-homogeneous} (where $n$ is finite) if every its nonzero irreducible
representation acts on an $n$-dimensional Hilbert space. An elementary proof of Fell's
characterization of $n$-homogeneous $C^*$-algebras (by means of their spectra) is presented.
A spectral theorem and a functional calculus for finite systems of elements which generate
$n$-homogeneous $C^*$-algebras are proposed.
\end{abstract}
\maketitle

\SECT{Introduction}

In 1961 Fell \cite{fe3} gave a characterization of $n$-homogeneous $C^*$-algebras in terms of fibre
bundles. It is a natural generalization of the commutative Gelfand-Naimark theorem, which gives
models for commutative $C^*$-algebras. However, Fell's proof involves the machinery of (general)
operator fields and as such is more advanced then Gelfand's theory of commutative Banach algebras.
In this paper we propose a new proof of his theorem (starting from the very beginning), which is
elementary and resembles the standard proof of the commutative Gelfand-Naimark theorem. We avoid
the abstract language of fibre bundles---instead of them we introduce $n$-spaces, which are
counterparts of locally compact Hausdorff spaces in the commutative case. These are locally compact
Hausdorff spaces endowed with a (continuous) free action of the group $\Uu_n = \UUu_n / Z(\UUu_n)$
where $\UUu_n$ is the unitary group of $n \times n$-matrices and $Z(\UUu_n)$ is its center.\par
Our approach to the subject mentioned above enables us to generalize the spectral theorem (for
a normal Hilbert space operator) to the context of finite systems generating homogeneous
$C^*$-algebras. It also allows building so-called $n$-functional calculus for such systems. These
and related topics are discussed in the present paper.\par
The paper is organized as follows. Section~2 is devoted to operator-valued version
of the Stone-Weierstrass theorem, which plays an important role in our proof of Fell's theorem
on homogeneous $C^*$-algebras (presented is Section~5). In Section~3 we define and establish basic
properties of so-called $n$-spaces $(X,.)$ (which, in fact, are the same as Fell's fibre bundles)
and corresponding to them $C^*$-algebras $C^*(X,.)$. These investigations are continued in the next
part where we define spectral $n$-measures and characterize by means of them \textit{all}
representations of $C^*(X,.)$ for any $n$-space $(X,.)$. In Section~5 we give a new proof of Fell's
characterization of homogeneous $C^*$-algebras. In the last, sixth, part we formulate the spectral
theorem for finite systems of elements which generate $n$-homogeneous $C^*$-algebras and build
the $n$-functional calculus for them.

\subsection*{Notation and terminology} If a $C^*$-algebra $\AAa$ has a unit $e$, the spectrum of $x$
is denoted by $\sigma(x)$ and it is the set of all $\lambda \in \CCC$ for which $x - \lambda e$ is
noninvertible in $\AAa$. For two selfadjoint elements $a$ and $b$ of $\AAa$ we write $a \leqsl b$
provided $b - a$ is nonnegative. If $a \leqsl b$ and $b - a$ is invertible in $\AAa$, we shall
express this by writing $a < b$ or $b > a$. The $C^*$-algebra of all bounded operators
on a (complex) Hilbert space $\HHh$ is denoted by $\BBb(\HHh)$. Representations of unital
$C^*$-algebra need not preserve unities and they are understood as $*$-homomorphisms into
$\BBb(\HHh)$ for some Hilbert space $\HHh$. A representation of a $C^*$-algebra is $n$-dimensional
if it acts on an $n$-dimensional Hilbert space. A \textit{map} is a continuous function.

\SECT{Operator-valued Stone-Weierstrass theorem}

The classical Stone-Weierstrass theorem finds many applications in functional analysis and
approximation theory. It reached many generalizations as well --- see e.g. \cite{gl1}, \cite{bis},
\cite{lon}, \cite{pop}, \cite{g-m}, \cite{hof} or \cite{tim} and the references there (consult also
Theorem~1.4 in \cite{fe3}, \S4.7 in \cite{sak} and Corollary~11.5.3 in \cite{dix}). A first
significant counterpart of it for general $C^*$-algebras was established by Glimm \cite{gl1}. Much
later Longo \cite{lon} and Popa \cite{pop} proved independently a stronger version of Glimm's
result, solving a long-standing problem in theory of $C^*$-algebras. In comparison to the classical
Stone-Weierstrass theorem or, for example, to its generalization by Timofte \cite{tim}, Glimm's and
Longo's-Popa's theorems are not settled in function spaces. In this section we propose another
version of the theorem under discussion which takes place in spaces of functions taking values
in $C^*$-algebras. As such, it may be considered as its very natural generalization. Although
the results of Glimm and Longo and Popa are stronger and more general than ours, they involve
advanced machinery of $C^*$-algebras and advanced language of this theory, while our approach is
very elementary and its proof is similar to Stone's \cite{st1,st2}. To formulate our result, we need
to introduce the following notion.

\begin{dfn}{sp-sep}
Let $X$ be a set, $x$ and $y$ be distinct points of $X$ and let $\AAa$ be a unital $C^*$-algebra.
A collection $\FfF$ of functions from $X$ to $\AAa$ \textit{spectrally separates} points $x$ and $y$
if there is $f \in \FfF$ such that $f(x)$ and $f(y)$ are normal elements of $\AAa$ and their spectra
are disjoint. If $\FfF$ spectrally separates any two distinct points of $X$, we say that $\FfF$
\textit{spectrally separates points} of $X$.
\end{dfn}

The reader should notice that a collection of complex-valued functions spectrally separates two
points iff it separates them.\par
Whenever $\AAa$ is a unital $C^*$-algebra and $a$ is a selfadjoint element of $\AAa$, let us
denote by $M(a)$ the real number $\max \sigma(a)$. Further, if $X$ is a locally compact Hausdorff
space and $f\dd X \to \AAa$ is a map, we say that $f$ \textit{vanishes at infinity} iff for every
$\epsi > 0$ there is a compact set $K \subset X$ such that $\|f(x)\| < \epsi$ for any $x \in X
\setminus K$. The set of all $\AAa$-valued maps on $X$ vanishing at infinity is denoted
by $\CCc_0(X,\AAa)$. Notice that $\CCc_0(X,\AAa)$ is a $C^*$-algebra when it is equipped with
pointwise actions and the supremum norm induced by the norm of $\AAa$. Moreover, $\CCc_0(X,\AAa)$ is
unital iff $X$ is compact (recall that we assume here that $\AAa$ is unital).\par
A full version of our Stone-Weierstrass type theorem has the following form.

\begin{thm}{SW-full}
Let $X$ be a locally compact Hausdorff space and let $\AAa$ be a unital $C^*$-algebra. Let $\EeE$
be a $*$-subalgebra of $\CCc_0(X,\AAa)$ such that:
\begin{enumerate}[\upshape({A}X1)]\addtocounter{enumi}{-1}
\item if $X$ is noncompact, then for each $z \in X$ either $f_0(z)$ is invertible in $\AAa$ for some
   $f_0 \in \EeE$ or $f(z) = 0$ for any $f \in \EeE$;
\end{enumerate}
and for any two points $x$ and $y$ of $X$ one of the following two conditions is fulfilled:
\begin{enumerate}[\upshape({A}X1)]
\item either $x$ and $y$ are spectrally separated by $\EeE$, or
\item $M(f(x)) = M(f(y))$ for any $f \in \EeE$.
\end{enumerate}
Then the \textup{(}uniform\textup{)} closure of $\EeE$ in $\CCc_0(X,\AAa)$ coincides with
the $*$-algebra $\Delta_2(\EeE)$ of all maps $u \in \CCc_0(X,\AAa)$ such that for any $x, y \in X$
and each $\epsi > 0$ there exists $v \in \EeE$ with $\|v(z) - u(z)\| < \epsi$ for $z \in \{x,y\}$.
\end{thm}

As a consequence of the above result we obtain the following result, which is a special case
of Corollary~11.5.3 in \cite{dix}.

\begin{pro}{SW-classic}
Let $X$ be a locally compact Hausdorff space and let $\AAa$ be a unital $C^*$-algebra.
A $*$-subalgebra $\EeE$ of $\CCc_0(X,\AAa)$ is dense in $\CCc_0(X,\AAa)$ iff $\EeE$ spectrally
separates points of $X$ and for every $x \in X$ the set $\EeE(x) := \{f(x)\dd\ f \in \EeE\}$ is
dense in $\AAa$.
\end{pro}

It is worth noting that we know no characterization of dense $*$-subalgebras of $\CCc_0(X,\AAa)$
in case $\AAa$ does not have a unit.\par
The proof of \THM{SW-full} is partially based on the original proof of the Stone-Weierstrass theorem
given by Stone \cite{st1,st2}. However, the key tool in our proof is the so-called Loewner-Heinz
inequality (for the discussion on this inequality, see page~150 in \cite{bha}), first proved
by Loewner \cite{low}:

\begin{thm}{l-h}
Let $a$ and $b$ be two selfadjoint nonnegative elements in a $C^*$-algebra such that $a \leqsl
b$. Then for every $s \in (0,1)$, $a^s \leqsl b^s$.
\end{thm}

The proof of \THM{SW-full} is preceded by several auxiliary results. For simplicity, the unit
of $\AAa$ will be denoted by $1$ and the function from $X$ to $\AAa$ constantly equal to $1$ will be
denoted by $1_X$. We also preserve the notation of \THM{SW-full}. Additionally, $\bar{\EeE}$ stands
for the (uniform) closure of $\EeE$ in $\CCc_0(X,\AAa)$.

\begin{lem}{0}
Suppose $X$ is compact. Then $1_X \in \bar{\EeE}$ \iaoi{} for every $x \in X$ there is $f \in \EeE$
such that $f(x)$ is invertible in $\AAa$.
\end{lem}
\begin{proof}
The necessity is clear (since the set of all invertible elements is open in $\AAa$). To prove
the sufficiency, for each $x \in X$ take $f_x \in \EeE$ such that $f_x(x)$ is invertible. Put $u_x =
f_x^* f_x \in \EeE$ and let $V_x \subset X$ consist of all $y \in X$ such that $u_x(y) > 0$.
It follows from the continuity of $u_x$ that $V_x$ is open. By the compactness of $X$, $X =
\bigcup_{j=1}^p V_{x_j}$ for some finite system $x_1,\ldots,x_p$. Put $u = \sum_{j=1}^p u_{x_j} \in
\EeE$ and note that $u(x) > 0$ for each $x \in X$. This implies that $u$ is invertible
in $\CCc_0(X,\AAa)$. Let $f\dd [0,\|u\|] \to \RRR$ be a map with $f(0) = 0$ and $f\bigr|_{\sigma(u)}
\equiv 1$. There is a sequence of real polynomials $p_1,p_2,\ldots$ which converge uniformly to $f$
on $[0,\|u\|]$. Then $p_n(u) \to f(u) = 1_X$ (in the norm topology) and hence $1_X \in \bar{\EeE}$.
\end{proof}

\begin{lem}{1}
Suppose $X$ is compact and $1_X \in \bar{\EeE}$. Let $x \in X$ and $\delta > 0$ be arbitrary. For
any selfadjoint $f \in \Delta_2(\EeE)$ there are selfadjoint $g, h \in \bar{\EeE}$ such that $g(x) =
f(x) = h(x)$ and $g - \delta \cdot 1_X \leqsl f \leqsl h + \delta \cdot 1_X$.
\end{lem}
\begin{proof}
It follows from the definition of $\Delta_2(\EeE)$ (and the fact that $*$-homomorphisms between
$C^*$-algebras have closed ranges) that for every $y \in X$ there is $f_y \in \bar{\EeE}$ with
$f_y(z) = f(z)$ for $z \in \{x,y\}$. Replacing, if needed, $f_y$ by $\frac12(f_y + f_y^*)$, we may
assume that $f_y$ is selfadjoint. Let $U_y \subset X$ consist of all $z \in X$ such that
$\|f_y(z) - f(z)\| < \delta$. Take a finite number of points $x_1,\ldots,x_p$ for which $X =
\bigcup_{j=1}^p U_{x_j}$. For simplicity, put $V_j = U_{x_j}$ and $g_j = f_{x_j}\ (j=1,\ldots,p)$.
Observe that $f - \delta \cdot 1_X \leqsl g_j$ on $V_j$ and $g_j(x) = f(x)$. We define by induction
functions $h_1,\ldots,h_p \in \bar{\EeE}$: $h_1 = g_1$ and $h_k = \frac12(h_{k-1} + g_k + |h_{k-1}
- g_k|)$ for $k = 2,\ldots,p$ where $|u| = \sqrt{u^* u}$ for each $u \in \bar{\EeE}$. Since
$\bar{\EeE}$ is a $C^*$-algebra, we clearly have $h_k \in \bar{\EeE}$. Use induction to show that
$h_j(x) = f(x)$ and $g_j \leqsl h_p$ for $j = 1,\ldots,p$. Then $h = h_p$ is the function
we searched for. Indeed, $h(x) = f(x)$ and for any $y \in X$ there is $j \in \{1,\ldots,p\}$ such
that $y \in V_j$, which implies that $f(y) - \delta \cdot 1 \leqsl g_j(y) \leqsl h(y)$.\par
Now if we apply the above argument to the function $-f$, we shall obtain a selfadjoint function
$h' \in \bar{\EeE}$ such that $-f(x) = h'(x)$ and $-f \leqsl h' + \delta \cdot 1_X$. Then put $g :=
- h'$ to complete the proof.
\end{proof}

\begin{lem}{2}
Let $\epsi > 0$, $r > 0$ and $k \geqsl 1$ be given. There is a natural number $N = N(\epsi,r,k)$
with the following property. If $a_1,\ldots,a_k,b$ are selfadjoint elements of $\AAa$ such that
$0 \leqsl a_j \leqsl b$, $b a_j = a_j b\ (j=1,\ldots,k)$ and $\|b\| \leqsl r$, then $a_s \leqsl
(\sum_{j=1}^k a_j^n)^{\frac1n} \leqsl b + \epsi \cdot 1$ for any $s \in \{1,\ldots,k\}$ and
$n \geqsl N$.
\end{lem}
\begin{proof}
Let $N \geqsl 2$ be such that $\sqrt[n]{k} \leqsl 1 + \frac{\epsi}{r}$ for each $n \geqsl N$ and let
$a_1,\ldots,a_k,b$ be as in the statement of the lemma. Since then $a_s^n \leqsl \sum_{j=1}^k
a_j^n$, \THM{l-h} yields $a_s \leqsl (\sum_{j=1}^k a_j^n)^{\frac1n}$. Further, since $b$ commutes
with $a_j$, we get $a_j^n \leqsl b^n$ and consequently $\sum_{j=1}^k a_j^n \leqsl k b^n$. So,
another application of \THM{l-h} gives us $(\sum_{j=1}^k a_j^n)^{\frac1n} \leqsl \sqrt[n]{k} b$. So,
it suffices to have $\sqrt[n]{k} b \leqsl b + \epsi \cdot 1$ which is fulfilled for $n \geqsl N$
because $\|(\sqrt[n]{k} - 1) b\| \leqsl (\sqrt[n]{k} - 1) r \leqsl \epsi$.
\end{proof}

\begin{lem}{3}
Suppose $X$ is compact and $1_X \in \bar{\EeE}$. If $f \in \Delta_2(\EeE)$ commutes with every
member of $\EeE$, then $f \in \bar{\EeE}$.
\end{lem}
\begin{proof}
Since $\Delta_2(\EeE)$ is a $*$-algebra, we may assume that $f$ is selfadjoint. Fix $\delta > 0$.
By \LEM{1}, for every $x \in X$ there is $f_x \in \bar{\EeE}$ with $f_x(x) = f(x)$ and $f_x \leqsl
f + \delta \cdot 1_X$. Let $U_x \subset X$ consist of all $y \in X$ such that $f_x(y) > f(y)
- \delta \cdot 1$. We infer from the compactness of $X$ that $X = \bigcup_{j=1}^k U_{x_j}$ for some
points $x_1,\ldots,x_k \in X$. For simplicity, we put $V_j = U_{x_j}$ and $g_j = f_{x_j}$. We then
have
\begin{equation}\label{eqn:1}
g_j(x) \geqsl f(x) - \delta \cdot 1 \quad \textup{ for any } x \in V_j
\end{equation}
and
\begin{equation}\label{eqn:2}
g_j(x) \leqsl f(x) + \delta \cdot 1 \quad \textup{ for any } x \in X.
\end{equation}
It follows from the compactness of $X$ that there is a constant $c > 0$ such that $g_j + c \cdot 1_X
\geqsl 0\ (j=1,\ldots,k)$ and $f + (c - \delta) \cdot 1_X \geqsl 0$. Further, there is $r > 0$ such
that $f + (c + \delta) \cdot 1_X \leqsl r \cdot 1_X$. Now let $N = N(\delta,r,k)$ be as in \LEM{2}.
Since $f$ commutes with each member of $\bar{\EeE}$, we conclude from that lemma and from
\eqref{eqn:2} that $g_s(x) + c \cdot 1 \leqsl [\sum_{j=1}^k (g_j(x) + c \cdot 1)^n]^{\frac1n} \leqsl
f(x) + (c + 2 \delta) \cdot 1$ for any $x \in X$. Finally, since $1_X \in \bar{\EeE}$, the function
$g := [\sum_{j=1}^k (g_j + c \cdot 1_X)^n]^{\frac1n} - c \cdot 1_X$ belongs to $\bar{\EeE}$. What is
more, $g \leqsl f + 2\delta \cdot 1_X$ and $g(x) \geqsl g_j(x) \geqsl f(x) - \delta \cdot 1$ for
$x \in V_j$ (cf. \eqref{eqn:1}). This gives $f - \delta \cdot 1_X \leqsl g$ on the whole space $X$
and therefore $-\delta \cdot 1_X \leqsl g - f \leqsl 2\delta \cdot 1_X$, which is equivalent
to $\|g - f\| \leqsl 2\delta$ and finishes the proof.
\end{proof}

\begin{lem}{4}
Suppose $X$ is compact, $1_X \in \bar{\EeE}$ and there exists an equivalence relation $\RRr$ on $X$
such that two points $x$ and $y$ are spectrally separated by $\EeE$ whenever $(x,y) \notin \RRr$.
Then every map $g\dd X \to \CCC \cdot 1 \subset \AAa$ which is constant on each equivalence class
with respect to $\RRr$ belongs to $\bar{\EeE}$.
\end{lem}
\begin{proof}
By \LEM{3}, we only need to check that $g \in \Delta_2(\EeE)$. We may assume that $g\dd X \to \RRR
\cdot 1$. Let $x$ and $y$ be arbitrary. Write $g(x) = \alpha \cdot 1$ and $g(y) = \beta \cdot 1$.
If $(x,y) \in \RRr$, then both $x$ and $y$ belong to the same equivalence class and hence $\alpha =
\beta$. Then $g(z) = (\alpha \cdot 1_X)(z)$ for $z \in \{x,y\}$ (and $\alpha \cdot 1_X \in
\bar{\EeE}$). Now assume that $(x,y) \notin \RRr$. Then---by assumption---there is $f \in \EeE$ such
that both $f(x)$ and $f(y)$ are normal and $\sigma(f(x)) \cap \sigma(f(y)) = \varempty$. Let
$\varphi\dd \CCC \to \RRR$ be a map such that $\varphi\bigr|_{\sigma(f(x))} \equiv \alpha$ and
$\varphi\bigr|_{\sigma(f(y))} \equiv \beta$. There is a sequence of polynomials
$p_1(z,\bar{z}),p_2(z,\bar{z}),\ldots$ which converge uniformly to $\varphi$ on $K := \sigma(f(x))
\cup \sigma(f(y))$. Then $p_n(f,f^*) \in \EeE$ and, for $w \in \{x,y\}$, $[p_n(f,f^*)](w) =
p_n(f(w),[f(w)]^*)$. Since $f(w)$ is normal and its spectrum is contained in $K$, we see that
$\lim_{n\to\infty} [p_n(f,f^*)](w) = \varphi(f(w))$. Now the notice that $\varphi(f(x)) = \alpha
\cdot 1 = g(x)$ and $\varphi(f(y)) = \beta \cdot 1 = g(y)$ finishes the proof.
\end{proof}

We recall that if $X$ is a compact Hausdorff space and $\RRr$ is a closed equivalence relation
on $X$, then the quotient topological space $X / \RRr$ is Hausdorff as well.

\begin{lem}{5}
Suppose $X$ is compact and there is a closed equivalence relation $\RRr$ on $X$ such that $M(f(x)) =
M(f(y))$ for each selfadjoint $f \in \EeE$ whenever $(x,y) \in \RRr$. Let $\pi\dd X \to X / \RRr$
denote the canonical projection, $f \in \bar{\EeE}$ be selfadjoint, $a$ and $b$ be two real numbers
and let $U = \{x \in X\dd\ a \cdot 1 < f(x) < b \cdot 1\}$. Then $\pi^{-1}(\pi(U)) = U$ and $\pi(U)$
is open in $X / \RRr$.
\end{lem}
\begin{proof}
Recall that $\pi(U)$ is open in $X / \RRr$ iff $\pi^{-1}(\pi(U))$ is open in $X$. Therefore
it suffices to show that $\pi^{-1}(\pi(U)) = U$. Of course, the inclusion `$\supset$' is immediate.
And if $y \in \pi^{-1}(\pi(U))$, then there is $x \in U$ such that $(x,y) \in \RRr$. We then have
$a \cdot 1 < f(x) < b \cdot 1$, $M(f(x)) = M(f(y))$ and $M(-f(x)) = M(-f(y))$ (the last two
relations follows from the fact that $f \in \bar{\EeE}$). The first of these relations says that
$[-M(-f(x)),M(f(x))] \subset (a,b)$ from which we infer that $[-M(-f(y)),M(f(y))] \subset (a,b)$ and
consequently $y \in U$.
\end{proof}

The following is a special case of \THM{SW-full}.

\begin{lem}{6}
Suppose $X$ is compact, $1_X \in \bar{\EeE}$ and for any $x, y \in X$ one of conditions
\textup{(AX1)--(AX2)} is fulfilled. Then $\Delta_2(\EeE) = \bar{\EeE}$.
\end{lem}
\begin{proof}
We only need to show that $\Delta_2(\EeE)$ is contained in $\bar{\EeE}$. Let $f \in \Delta_2(\EeE)$
be selfadjoint and let $\delta > 0$. We shall construct $w \in \bar{\EeE}$ such that $\|w - f\|
\leqsl 3 \delta$. By \LEM{1}, for each $x \in X$ there are functions $u_x, v_x \in \bar{\EeE}$ such
that $u_x(x) = f(x) = v_x(x)$ and $u_x - \delta \cdot 1_X < f < v_x + \delta \cdot 1_X$. Let $G_x
\subset X$ consist of all $y \in X$ such that $v_x(y) - \delta \cdot 1 < f(y) < u_x(y) + \delta
\cdot 1$. Since $x \in G_x$ and $X$ is compact, there is a finite system $x_1,\ldots,x_k \in X$ for
which $X = \bigcup_{j=1}^k G_{x_j}$. For simplicity, we put $W_j = G_{x_j}$, $p_j = u_{x_j}$ and
$q_j = v_{x_j}$. Observe that then
\begin{equation}\label{eqn:3}
p_j(x) - \delta \cdot 1 < f(x) < q_j(x) + \delta \cdot 1 \quad \textup{for any } x \in X
\end{equation}
and
\begin{equation}\label{eqn:4}
q_j(x) - \delta \cdot 1 < f(x) < p_j(x) + \delta \cdot 1 \quad \textup{for any } x \in W_j.
\end{equation}
Let $D_j$ consist of all $x \in X$ such that $-2\delta \cdot 1 < p_j(x) - q_j(x) < 2\delta \cdot 1$.
We infer from \eqref{eqn:3} and \eqref{eqn:4} that $W_j \subset D_j$ and thus $X = \bigcup_{j=1}^k
D_j$. Further, let $\RRr$ be an equivalence relation on $X$ given by the rule: $(x,y) \in \RRr \iff
M(u(x)) = M(u(y))$ for each selfadjoint $u \in \EeE$. It follows from the definition of $\RRr$ that
$\RRr$ is closed in $X \times X$. Denote by $\pi\dd X \to X / \RRr$ the canonical projection.
We deduce from \LEM{5} that the sets $\pi(D_1),\ldots,\pi(D_k)$ form an open cover of the space $X /
\RRr$ (which is compact and Hausdorff). Now let $\beta_1,\ldots,\beta_k\dd X / \RRr \to [0,1]$ be
a partition of unity such that $\beta_j^{-1}((0,1]) \subset \pi(D_j)$ for $j=1,\ldots,k$. Put
$\alpha_j = (\beta_j \circ \pi) \cdot 1\dd X \to \CCC \cdot 1 \subset \AAa$. \LEM{4} combined with
conditions (AX1)--(AX2) yields that $\alpha_1,\ldots,\alpha_k \in \bar{\EeE}$. Define $w \in
\bar{\EeE}$ by $w = \sum_{j=1}^k \alpha_j p_j$. Since $\sum_{j=1}^k \alpha_j = 1_X$, we conclude
from \eqref{eqn:3} that $w \leqsl f + \delta \cdot 1_X$. So, to end the proof, it is enough to check
that $f(x) \leqsl w(x) + 3\delta \cdot 1$ for each $x \in X$. This inequality will be satisfied
provided
\begin{equation}\label{eqn:5}
\alpha_j(x) (f(x) - 3\delta \cdot 1) \leqsl \alpha_j(x) p_j(x)
\end{equation}
for any $j$. We consider two cases. If $x \in D_j$, then $p_j(x) > q_j(x) - 2\delta \cdot 1 > f(x)
- 3\delta \cdot 1$ (by \eqref{eqn:3}) and consequently \eqref{eqn:5} holds. Finally, if $x \notin
D_j$, then $\pi(x) \notin \pi(D_j)$ (see \LEM{5}) and therefore $\alpha_j(x) = 0$, which easily
gives \eqref{eqn:5}.
\end{proof}

\begin{proof}[Proof of \THM{SW-full}]
We only need to check that $\Delta_2(\EeE) \subset \bar{\EeE}$. We consider two cases.\par
First assume $X$ is compact. Let $\EeE' = \EeE + \CCC \cdot 1_X$. Observe that $\EeE'$ is
a $*$-algebra and for any two points $x$ and $y$ one of conditions (AX1)--(AX2) is fulfilled with
$\EeE$ replaced by $\EeE'$. Consequently, it follows from \LEM{6} that $\overline{\EeE'} =
\Delta_2(\EeE')$. But $\overline{\EeE'} = \bar{\EeE} + \CCC \cdot 1_X$. So, for any $g \in
\Delta_2(\EeE)$ we clearly have $g \in \Delta_2(\EeE')$ and hence $g = f + \lambda \cdot 1_X$ for
some $f \in \bar{\EeE}$ and $\lambda \in \CCC$. If $\lambda = 0$, then $g = f \in \bar{\EeE}$ and
we are done. Otherwise, $1_X = \frac{1}{\lambda} (g - f) \in \Delta_2(\EeE)$ which implies that
the assumptions of \LEM{0} are satisfied. We infer from that lemma that $1_X \in \bar{\EeE}$ and
therefore $g \in \bar{\EeE}$ as well.\par
Now assume $X$ is noncompact. Let $\widehat{X} = X \cup \{\infty\}$ be the one-point
compactification of $X$. Every function $f \in \CCc_0(X,\AAa)$ admits a unique continuous extension
$\widehat{f}\dd \widehat{X} \to \AAa$, given by $\widehat{f}(\infty) = 0$. Denote
by $\widehat{\EeE}$ the $*$-subalgebra of $\CCc(\widehat{X},\AAa)$ constisting of all extensions
of (all) functions from $\EeE$. We claim that for any $x, y \in \widehat{X}$ one of conditions
(AX1)--(AX2) is fulfilled with $\EeE$ replaced by $\widehat{\EeE}$. Indeed, if both $x$ and $y$
differ from $\infty$, this follows from our assumptions about $\EeE$. And if, for example, $y =
\infty \neq x$, condition (AX0) implies that either $M(\widehat{f}(x)) = M(\widehat{f}(y))$ for each
$f \in \EeE$ or $\widehat{u}(x)$ is invertible in $\AAa$ for some $u \in \EeE$. But then $f = u^*u
\in \EeE$ is normal and $0 \notin \sigma(\widehat{f}(x))$, while $\sigma(\widehat{f}(y)) = \{0\}$,
which shows that $x$ and $y$ are spectrally separated by $\widehat{\EeE}$. So, it follows from
the first part of the proof that the closure of $\widehat{\EeE}$ in $\CCc(\widehat{X},\AAa)$
coincides with $\Delta_2(\widehat{\EeE})$. But the closure of $\widehat{\EeE}$ coincides with
$\{\widehat{f}\dd\ f \in \bar{\EeE}\}$ and $\Delta_2(\widehat{\EeE}) = \{\widehat{f}\dd\ f \in
\Delta_2(\EeE)\}$. We infer from these that $\Delta_2(\EeE) = \bar{\EeE}$ and the proof is complete.
\end{proof}

\begin{proof}[Proof of \PRO{SW-classic}]
The necessity of the condition is clear (since for any two distinct points $x$ and $y$ in $X$ and
any elements $a$ and $b$ of $\AAa$ there is a function $f \in \CCc_0(X,\AAa)$ such that $f(x) = a$
and $f(y) = b$). To prove the sufficiency, assume $\EeE$ spectrally separates points of $X$ and for
each $x \in X$ the set $\EeE(x)$ is dense in $\AAa$. First notice that then for each $x \in X$ there
is $f \in \EeE$ such that $f(x)$ is invertible in $\AAa$. This shows that all assumptions
of \THM{SW-full} are satisfied. According to that result, we only need to show that for any two
distinct points $x$ and $y$ of $X$ the set $L := \{(f(x),f(y))\dd\ f \in \EeE\}$ is dense in $\AAa
\times \AAa$. Since $x$ and $y$ are spectrally separated by $\EeE$, the proof of \LEM{4} shows that
$(1,0),(0,1) \in \bar{L}$. Further, since both $\EeE(x)$ and $\EeE(y)$ are dense in $\AAa$,
we conclude that $\{f(x)\dd\ f \in \bar{\EeE}\} = \{f(y)\dd\ f \in \bar{\EeE}\} = \AAa$ and
therefore for arbitrary two elements $a$ and $b$ of $\AAa$ there are $u,v \in \bar{\EeE}$ for which
$u(x) = a$ and $v(y) = b$. Then $(a,b) = (u(x),u(y)) \cdot (1,0) + (v(x),v(y)) \cdot (0,1) \in
\bar{L}$ (we use here the coordinatewise multiplication) and we are done.
\end{proof}

Taking into account \PRO{SW-classic}, the following question may be interesting.

\begin{prb}{1}
Let $\AAa$ be a $C^*$-algebra without unit and let $X$ be a locally compact Hausdorff space. Is
it true that a $*$-subalgebra $\EeE$ of $\CCc_0(X,\AAa)$ is dense in $\CCc_0(X,\AAa)$ iff for any
two points $x$ and $y$ of $X$ the set $\{(f(x),f(y))\dd\ f \in \EeE\}$ is dense in $\AAa \times
\AAa$~?
\end{prb}

\SECT{Topological $n$-spaces}

In Fell's characterization of homogeneous $C^*$-algebras \cite{fe3} (consult also Theorem~IV.1.7.23
in \cite{bla}) special fibre bundles appear. To make our lecture as simple and elementary
as possible, we avoid this language and instead of using fibre bundles we shall introduce so-called
\textit{$n$-spaces} (see \DEF{n-sp} below). To this end, let $M_n$ be the $C^*$-algebra of all
complex $n \times n$-matrices. Let $\UUu_n$ be the unitary group of $M_n$ and $I$ be its neutral
element. Let $\TTT = \{z \in \CCC\dd\ |z| = 1\}$. Let $\Uu_n$ denote the compact topological group
$\UUu_n / (\TTT \cdot I)$ and let $\pi_n\dd \UUu_n \to \Uu_n$ be the canonical homomorphism. Members
of $\Uu_n$ will be denoted by $\uU$ and $\vV$, and $\jJ$ is reserved for the neutral element
of $\Uu_n$. The (probabilistic) Haar measure on $\Uu_n$ will be denoted by $\dint{\uU}$. For any
$A \in M_n$ and $\uU \in \Uu_n$ let $\uU.A$ denote the matrix $U A U^{-1}$ where $U \in \UUu_n$ is
such that $\pi_n(U) = \uU$. It is easily seen that the function
$$\Uu_n \times M_n \ni (\uU,A) \mapsto \uU.A \in M_n$$
is a well defined continuous action of $\Uu_n$ on $M_n$ (which means that $\jJ.A = A$ and
$\uU.(\vV.A) = (\uU\vV).A$ for any $\uU,\vV \in \Uu_n$ and $A \in M_n$). More generally, for any
$C^*$-algebra $\AAa$ let $M_n(\AAa)$ be the algebra of all $n \times n$-matrices with entries
in $\AAa$. ($M_n(\AAa)$ may naturally be identified with $\AAa \otimes M_n$.) For any matrix $A \in
M_n(\AAa)$ and each $\uU \in \Uu_n$, $\uU.A$ is defined as $U A U^{-1}$ where $U \in \UUu_n$ is such
that $\pi_n(U) = \uU$, and $U A U^{-1}$ is computed in a standard manner.

\begin{dfn}{n-sp}
A pair $(X,.)$ is said to be an \textit{$n$-space} if $X$ is a locally compact Hausdorff space and
$\Uu_n \times X \ni (\uU,x) \mapsto \uU.x \in X$ is a continuous \textbf{free} action of $\Uu_n$
on $X$. Recall that the action is free iff the equality $\uU.x = x$ (for some $x \in X$) implies
$\uU = \jJ$.\par
Let $(X,.)$ be an $n$-space. Let $C^*(X,.)$ be the $*$-algebra of all maps $f \in \CCc_0(X,M_n)$
such that $f(\uU.x) = \uU.f(x)$ for any $\uU \in \Uu_n$ and $x \in X$. $C^*(X,.)$ is
a $C^*$-subalgebra of $\CCc_0(X,M_n)$.\par
By a \textit{morphism} between two $n$-spaces $(X,.)$ and $(Y,*)$ we mean any proper map $\psi\dd X
\to Y$ such that $\psi(\uU.x) = \uU * \psi(x)$ for any $\uU \in \Uu_n$ and $x \in X$. (A map is
proper if the inverse images of compact sets under this map are compact.) A morphism which is
a homeomorphism is said to be an \textit{isomorphism}. Two $n$-spaces are \textit{isomorphic}
if there exists an isomorphism between them.
\end{dfn}

The reader should notice that the (natural) action of $\Uu_n$ on $M_n$ is \underline{not} free.
However, one may check that the set $\Mm_n$ of all irreducible matrices $A \in M_n$ (that is, $A \in
\Mm_n$ iff every matrix $X \in M_n$ which commutes with both $A$ and $A^*$ is of the form
$\lambda I$ where $\lambda \in \CCC$) is open in $M_n$ (and thus $\Mm_n$ is locally compact) and
the action $\Uu_n \times \Mm_n \ni (\uU,A) \mapsto \uU.A \in \Mm_n$ is free, which means that
$(\Mm_n,.)$ is an $n$-space.\par
In this section we establish basic properties of $C^*$-algebras of the form $C^*(X,.)$ where $(X,.)$
is an $n$-space. To this end, recall that whenever $(\Omega,\Mm,\mu)$ is a finite measure space and
$f\dd \Omega \ni \omega \mapsto (f_1(\omega),\ldots,f_k(\omega)) \in \CCC^k$ is an $\Mm$-measurable
(which means that $f^{-1}(U) \in \Mm$ for every open set $U \subset \CCC^k$) bounded function, then
$\int_{\Omega} f(\omega) \dint{\mu(\omega)}$ is (well) defined as $(\int_{\Omega} f_1(\omega)
\dint{\mu(\omega)},\ldots,\int_{\Omega} f_k(\omega) \dint{\mu(\omega)})$. If $\|\cdot\|$ is any norm
on $\CCC^k$, then $\|\int_{\Omega} f(\omega) \dint{\mu(\omega)}\| \leqsl \int_{\Omega} \|f(\omega)\|
\dint{\mu(\omega)}$. In particular, the above rules apply to matrix-valued measurable functions.\par
From now on, $n \geqsl 1$ and an $n$-space $(X,.)$ are fixed. A set $A \subset X$ is said to be
\textit{invariant} provided $\uU.a \in A$ for any $\uU \in \Uu_n$ and $a \in A$. Observe that if $A$
is closed or open and $A$ is invariant, then $A$ is locally compact and consequently $(A,.)$ is
an $n$-space (when the action of $\Uu_n$ is restricted to $A$). We begin with

\begin{lem}{ext}
For each $f \in \CCc_0(X,M_n)$ let $f^{\Uu}\dd X \to M_n$ be given by:
$$f^{\Uu}(x) = \int_{\Uu_n} \uU^{-1}.f(\uU.x) \dint{\uU} \qquad (x \in X).$$
\begin{enumerate}[\upshape(a)]
\item For any $f \in \CCc_0(X,M_n)$, $f^{\Uu} \in C^*(X,.)$.
\item If $f \in \CCc_0(X,M_n)$ and $x \in X$ are such that $f(\uU.x) = \uU.f(x)$ for any $\uU \in
   \Uu_n$, then $f^{\Uu}(x) = f(x)$.
\item Let $A \subset X$ be a closed invariant nonempty set. Every map $g\in C^*(A,.)$ extends
   to a map $\widetilde{g} \in C^*(X,.)$ such that $\sup_{a \in A} \|g(a)\| = \sup_{x \in X}
   \|\widetilde{g}(x)\|$.
\item For any $x \in X$ and $A \in M_n$ there is $f \in C^*(X,.)$ with $f(x) = A$.
\item Let $x$ and $y$ be two points of $X$ such that there is no $\uU \in \Uu_n$ for which $\uU.x =
   y$. Then for any $A, B \in M_n$ there is $f \in C^*(X,.)$ such that $f(x) = A$ and $f(y) = B$.
\item $C^*(X,.)$ has a unit iff $X$ is compact.
\end{enumerate}
\end{lem}
\begin{proof}
It is clear that $f^{\Uu}$ is continuous for every $f \in \CCc_0(X,M_n)$. Further, if $K \subset X$
is a compact set such that $\|f(x)\| \leqsl \epsi$ for each $x \in X \setminus K$, then
$\|f^{\Uu}(z)\| \leqsl \epsi$ for any $z \in X \setminus \Uu_n.K$ where $\Uu_n.K = \{\uU.x\dd\ \uU
\in \Uu_n,\ x \in K\}$. The note that $\Uu_n.K$ is compact leads to the conclusion that $f^{\Uu} \in
\CCc_0(X,M_n)$. Finally, for any $\vV \in \Uu_n$, any representant $V \in \UUu_n$ of $\vV$ and each
$x \in X$, we have
\begin{multline*}
f^{\Uu}(\vV.x) = \int_{\Uu_n} \uU^{-1}.f(\uU\vV.x) \dint{\uU} = \int_{\Uu_n}
(\uU\vV^{-1})^{-1}.f(\uU.x) \dint{\uU}\\= \int_{\Uu_n} \vV.[\uU^{-1}.f(\uU.x)] \dint{\uU} =
\int_{\Uu_n} V [\uU^{-1}.f(\uU.x)] V^{-1} \dint{\uU}\\= V \cdot \Bigl(\int_{\Uu_n}
\uU^{-1}.f(\uU.x) \dint{\uU}\Bigr) \cdot V^{-1} = \vV.f^{\Uu}(x),
\end{multline*}
which proves (a). Point (b) is a simple consequence of the definition of $f^{\Uu}$. Further, if $g$
is as in (c), it follows from Tietze's type theorem that there is $G \in \CCc_0(X,M_n)$ which
extends $g$ and satisfies $\sup_{a \in A} \|g(a)\| = \sup_{x \in X} \|G(x)\|$ (if $X$ is noncompact,
consider the one-point compactification $\widehat{X} = X \cup \{\infty\}$ of $X$ and note that then
the set $\widehat{A} = A \cup \{\infty\}$ is closed in $\widehat{X}$ and $g$ extends continuously
to $\widehat{A}$\,). Then $\widetilde{g} = G^{\Uu}$ is a member of $C^*(X,.)$ (by (a)) which
we searched for (see (b)).\par
We turn to (d) and (e). Let $K = \Uu_n.\{x\}$ and $f_0\dd K \to M_n$ be given by $f_0(\uU.x) =
\uU.A\ (\uU \in \Uu_n)$. Since the action of $\Uu_n$ on $X$ is free, $f_0$ is a well defined map.
Since $K$ is compact, (c) yields the existence of $f \in C^*(X,.)$ which extends $f_0$. To prove
(e), we argue similarly: put $L = \Uu_n.\{x,y\}$ and let $g_0\dd L \to M_n$ be given by $g_0(\uU.x)
= \uU.A$ and $g_0(\uU.y) = \uU.B\ (\uU \in \Uu_n)$. We infer from the assumption of (e) that $g_0$
is a well defined map. Consequently, since $L$ is compact, there exists, by (c), a map $g \in
C^*(X,.)$ which extends $g_0$. This finishes the proof of (e), while point (f) immediately
follows from (d).
\end{proof}

\begin{pro}{repr}
\begin{enumerate}[\upshape(a)]
\item For every closed two-sided ideal $\IiI$ in $C^*(X,.)$ there exists
   a \textup{(}unique\textup{)} closed invariant set $A \subset X$ such that $\IiI$ coincides with
   the ideal $\IiI_A$ of all functions $f \in C^*(X,.)$ which vanish on $A$. Moreover, $C^*(X,.) /
   \IiI$ is ``naturally'' isomorphic to $C^*(A,.)$.
\item Let $k \leqsl n$ and let $\pi\dd C^*(X,.) \to M_k$ be a nonzero representation. Then $k =
   n$ and there is a unique point $x \in X$ such that $\pi(f) = f(x)$ for $f \in C^*(X,.)$.
\item Let $(Y,*)$ be an $n$-space. For every $*$-homomorphism $\Phi\dd (X,.) \to (Y,*)$ there is
   a unique pair $(U,\varphi)$ where $U$ is an open invariant subset of $Y$, $\varphi\dd (U,*) \to
   (X,.)$ is a morphism of $n$-spaces and
   \begin{equation}\label{eqn:homo}
   [\Phi(f)](y) = \begin{cases}
                  f(\varphi(y)) & \textup{if } y \in U\\
                  0             & \textup{if } y \notin U.
                  \end{cases}
   \end{equation}
   In particular, $C^*(X,.)$ and $C^*(Y,*)$ are isomorphic iff so are $(X,.)$ and $(Y,*)$.
\end{enumerate}
\end{pro}
\begin{proof}
The uniqueness of $A$ in (a) follows from point (e) of \LEM{ext}. To show its existence, let $A$
consist of all $x \in X$ such that $f(x) = 0$ for any $f \in \IiI$. It is clear that $A$ is closed
and invariant, and that $\IiI \subset \IiI_A$. To prove the converse inclusion, we shall involve
\THM{SW-full} for $\EeE = \IiI$. First of all, it follows from point (d) of \LEM{ext} that for each
$x \in X$ the set $\IiI(x) := \{f(x)\dd\ f \in \IiI\}$ is a two-sided ideal in $M_n$. Since $\{0\}$
is the only proper ideal of $M_n$, we conclude that $\IiI(x) = \{0\}$ for $x \in A$ and $\IiI(x) =
M_n$ for $x \in X \setminus A$. This shows that condition (AX0) of \THM{SW-full} is satisfied.
Further, if $x$ and $y$ are arbitrary points of $X$, then either
\begin{itemize}
\item $x, y \in A$; in that case (AX2) is fulfilled; or
\item $x \in A$ and $y \notin A$ (or conversely); in that case there is $f \in \IiI$ such that $f(y)
   = I$, and $f(x) = 0$ (since $x \in A$)---this implies that $x$ and $y$ are spectrally separated
   by $\IiI$; or
\item $x, y \notin A$ and $y = \uU.x$ for some $\uU \in \Uu_n$; in that case (AX2) is fulfilled
   since for any selfadjoint $f \in \IiI$, $f(y) = \uU.f(x)$ and consequently $\sigma(f(x)) =
   \sigma(f(y))$; or
\item $x, y \notin A$ and $y \notin \Uu_n.\{x\}$; in that case there are $f_1 \in \IiI$ and $f_2 \in
   C^*(X,.)$ such that $f_1(x) = I = f_2(x)$ and $f_2(y) = 0$ (cf. point (e) of \LEM{ext}), then
   $f = f_1 f_2 \in \IiI$ is such that $f(x) = I$ and $f(y) = 0$ and hence $x$ and $y$ are
   spectrally separated by $\IiI$.
\end{itemize}
Now according to \THM{SW-full}, it suffices to check that $\IiI_A \subset \Delta_2(\IiI)$ (since
$\IiI$ is closed). To this end, we fix $f \in \IiI_A$ and two arbitrary points $x$ and $y$ of $X$.
We consider similar cases as above:
\begin{enumerate}[(1$^{\circ}$)]
\item If $x, y \in A$, we have nothing to do because then $f(x) = f(y) = 0$.
\item If $x \in A$ and $y \notin A$ (or conversely), then there is $g \in \IiI$ such that $g(y) =
   f(y)$. But also $g(x) = 0 = f(x)$ and we are done.
\item If $x, y \notin A$ and $y = \uU.\{x\}$ for some $\uU \in \Uu_n$, then there is $g \in \IiI$
   with $g(x) = f(x)$. Then also $g(y) = g(\uU.x) = \uU.g(x) = \uU.f(x) = f(y)$ and we are done.
\item Finally, if $x,y \notin A$ and $y \notin \Uu_n.\{x\}$, there are functions $g_1,g_2 \in \IiI$
   and $h_1,h_2 \in C^*(X,.)$ such that $g_1(x) = f(x)$, $g_2(y) = f(y)$, $h_1(x) = I = h_2(y)$
   and $h_1(y) = 0 = h_2(x)$. Then $g = g_1 h_1 + g_2 h_2 \in \IiI$ satisfies $g(z) = f(z)$ for
   $z \in \{x,y\}$.
\end{enumerate}
The arguments (1$^{\circ}$)--(4$^{\circ}$) show that $f \in \Delta_2(\IiI)$ and thus $\IiI =
\IiI_A$. It follows from point (c) of \LEM{ext} that the $*$-homomorphism $C^*(X,.) \ni f \mapsto
f\bigr|_A \in C^*(A,.)$ is surjective. What is more, its kernel coincides with $\IiI_A = \IiI$
and therefore $C^*(X,.)$ and $C^*(A,.)$ are isomorphic.\par
We now turn to (b). We infer from (a) that there is a closed invariant set $A \subset X$ such that
$\ker(\pi) = \IiI_A$. Since $\pi$ is nonzero, $A$ is nonempty. Further, $k^2 \geqsl \dim
\pi(C^*(X,.)) = \dim(C^*(X,.) / \ker(\pi)) = \dim C^*(A,.) \geqsl n^2$ (by point (d)
of \LEM{ext} and by (a)) and thus $k = n$, $\dim C^*(A,.) = n^2$ and $\pi$ is surjective. Fix
$a \in A$ and observe that $A = \Uu_n.\{a\}$, because otherwise $\dim C^*(A,.) > n^2$ (thanks
to point (e) of \LEM{ext}). Now define $\Phi\dd M_n \to M_n$ by the rule $\Phi(X) = f(a)$ where
$\pi(f) = X$. It may easily be checked (using the fact that $\ker(\pi) = \IiI_{\Uu_n.\{a\}}$) that
$\Phi$ is a well defined one-to-one $*$-homomorphism of $M_n$. We conclude that there is $\uU \in
\Uu_n$ for which $\Phi(X) = \uU.X$ (in the algebra of matrices this is quite an elementary fact;
however, this follows also from \cite[Corollary~2.9.32]{sak}). Put $x = \uU^{-1}.a$ and note that
then $f(a) = \Phi(\pi(f)) = \uU.\pi(f)$ and consequently $\pi(f) = \uU^{-1}.f(a) = f(x)$ for each
$f \in C^*(X,.)$. The uniqueness of $x$ follows from points (d) and (e) of \LEM{ext}.\par
We turn to (c). Let $\Phi\dd C^*(X,.) \to C^*(Y,*)$ be a $*$-homomorphism
of $C^*$-algebras. Put $U = Y \setminus \{y \in Y\dd\ [\Phi(f)](y) = 0 \textup{ for each } f \in
C^*(X,.)\}$. It is clear that $U$ is invariant and open in $Y$. For any $y \in U$ the function
$C^*(X,.) \ni f \mapsto [\Phi(f)](y) \in M_n$ is a nonzero representation and therefore, thanks
to (b), there is a unique point $\varphi(y) \in X$ such that $[\Phi(f)](y) = f(\varphi(x))$ for each
$f \in C^*(X,.)$. In this way we have obtained a function $\varphi\dd U \to X$ for which
\eqref{eqn:homo} holds. By the uniqueness in (b), we see that $\varphi(\uU.y) = \uU.\varphi(y)$ for
any $\uU \in \Uu_n$ and $y \in U$. So, to prove that $\varphi$ is a morphism of $n$-spaces,
it remains to check that $\varphi$ is a proper map. First we shall show that $\varphi$ is
continuous. Suppose, for the contrary, that there is a set $D \subset U$ and a point $b \in U \cap
\bar{D}$ ($\bar{D}$ is the closure of $D$ in $Y$) such that $a := \varphi(b) \notin
\overline{\varphi(D)}$ (the closure taken in $X$). Let $V$ be an open neighbourhood of $a$ whose
closure is compact and disjoint from $F := \overline{\varphi(D)}$. Let $\scalarr$ be the standard
inner product on $M_n$, that is, $\scalar{X}{Y} = \tr(Y^*X)$ (`$\tr$' is the trace) and let $\|X\|_2
:= \sqrt{\tr(X^*X)}$. Take an irreducible matrix $Q \in M_n$ with $\|Q\|_2 = 1$. For simplicity, put
$\BBb = \{X \in M_n\dd\ \|X\|_2 \leqsl 1\}$. Our aim is to construct $f \in C^*(X,.)$ such that
$f(a) = Q$ and $f^{-1}(\{Q\}) \subset V$. Observe that there is a compact convex nonempty set $\KKk$
such that
\begin{equation}\label{eqn:12}
Q \notin \KKk \subset \BBb \quad \textup{and} \quad \{\uU.a\dd\ \uU \in \Uu_n,\ \uU.Q \notin \KKk\}
\subset V.
\end{equation}
(Indeed, it suffices to define $\KKk$ as the convex hull of the set $\{X \in \BBb\dd\ \|X - Q\|_2
\geqsl r\}$ where $r > 0$ is such that $\uU.a \in V$ whenever $\uU \in \Uu_n$ satisfies
$\|\uU.Q - Q\|_2 < r$. Such $r$ exists because $Q$ is irreducible and hence the maps $\Uu_n \ni \uU
\mapsto \uU.b \in X$ and $\Uu_n \ni \uU \mapsto \uU.Q \in M_n$ are embeddings.) Let $W =
\Uu_n.\{a\}$ and let $g_0\dd W \to M_n$ be given by $g_0(\uU.a) = \uU.Q$. Since $g_0(W \setminus V)
\subset \KKk$ (by \eqref{eqn:12}) and the set $\KKk$ (being compact, convex and nonempty) is
a retract of $M_n$, there is a map $g_1 \in \CCc_0(X \setminus V,M_n)$ such that $g_1(X \setminus V)
\subset \KKk$ and $g_1(x) = g_0(x)$ for $x \in W \setminus V$. Finally, there is $g \in
\CCc_0(X,M_n)$ which extends both $g_0$ and $g_1$, and $g(X) \subset \BBb$. Now put $f = g^{\Uu} \in
C^*(X,.)$ and notice that $f(a) = Q$ (by point (b) of \LEM{ext}). We claim that
\begin{equation}\label{eqn:13}
f^{-1}(\{Q\}) \subset V.
\end{equation}
Let us prove the above relation. Let $x \in X \setminus V$. Then $g(x) = g_1(x) \in \KKk$ and hence
$g(x) \neq Q$ (see \eqref{eqn:12}). The set $\Gg := \{\uU \in \Uu_n\dd\ \uU^{-1}.g(\uU.x) \neq Q\}$
is open in $\Uu_n$ and nonempty, which implies that its Haar measure is positive. Further,
$|\scalar{\uU^{-1}.g(\uU.x)}{Q}| \leqsl 1$ for any $\uU \in \Uu_n$ and
$\scalar{\uU^{-1}.g(\uU.x)}{Q} \neq 1$ for $\uU \in \Gg$ (since $g(X) \subset \BBb$). We infer from
these remarks that $\int_{\Uu_n} \scalar{\uU^{-1}.g(\uU.x)}{Q} \dint{\uU} \neq 1$. Equivalently,
$\scalar{f(x)}{Q} \neq 1$, which implies that $f(x) \neq Q$ and finishes the proof
of \eqref{eqn:13}. For $m \geqsl 1$ let $C_m = \{y \in Y\dd\ \|[\Phi(f)](y) - Q\|_2 \leqsl 2^{-m}\}$
and $F_m = \{x \in X\dd\ \|f(x) - Q\|_2 \leqsl 2^{-m}\}$. Since $f \in \CCc_0(X,M_n)$ and $\Phi(f)
\in \CCc_0(Y,M_n)$, $F_m$ is compact and $C_m$ is a compact neighbourhood of $b$. Consequently,
$C_m \cap D \neq \varempty$. We infer from \eqref{eqn:homo} that $\varphi(C_m \cap D) \subset F_m
\cap F$. Now the compactness argument gives $F \cap \bigcap_{m=1}^{\infty} F_m \neq \varempty$. Let
$c$ belong to this intersection. Then $f(c) = Q$ and $c \notin V$, which contradicts \eqref{eqn:13}
and finishes the proof of the continuity of $\varphi$.\par
To see that $\varphi$ is proper, take a compact set $K \subset X$ and note that $L = \Uu_n.K$ is
compact as well. Let $G \subset X$ be an open neighbourhood of $L$ with compact closure. Take a map
$\beta \in \CCc_0(X,M_n)$ such that $\beta(x) = I$ for $x \in L$ and $\beta$ vanishes off $G$. Let
$f = \beta^{\Uu} \in C^*(X,.)$ and observe that $f(x) = I$ for $x \in L$. Since $\Phi(f) \in
\CCc_0(Y,M_n)$, the set $Z := \{y \in Y\dd\ [\Phi(f)](y) = I\}$ is a compact subset of $Y$. But
\eqref{eqn:homo} implies that $Z \subset U$ and $\varphi^{-1}(K) \subset Z$. This finishes the proof
of the fact that $\varphi$ is a morphism. The uniqueness of the pair $(U,\varphi)$ follows from
\LEM{ext} and is left to the reader.\par
Now if $\Phi$ is a $*$-isomorphism of $C^*$-algebras, then $U = Y$ (by point (d) of \LEM{ext})
and thus $\Phi(f) = f \circ \varphi$. Similarly, $\Phi^{-1}$ is of the form $\Phi^{-1}(g) = g \circ
\psi$ for some morhism $\psi\dd (X,.) \to (Y,*)$. Then $f = f \circ (\varphi \circ \psi)$ for each
$f \in C^*(X,.)$ and the uniqueness in (c) gives $(\varphi \circ \psi)(x) = x$ for each $x \in
X$. Similarly, $(\psi \circ \varphi)(y) = y$ for any $y \in Y$ and consequently $\varphi$ is
an isomorphism of $n$-spaces. The proof is complete.
\end{proof}

\SECT{Representations of $C^*(X,.)$}

In this section we will characterize all representations of $C^*(X,.)$ for an arbitrary $n$-space
$(X,.)$. But first we shall give a `canonical' description of all continuous linear functionals
on $C^*(X,.)$. We underline here that we are not interested in the formula for the norm
of a functional. The results of the section will be applied in the next two parts where we formulate
our version of Fell's characterization of homogeneous $C^*$-algebras (Section~5) and a counterpart
of the spectral theorem for finite systems of operators which generate $n$-homogeneous
$C^*$-algebras (Section~6).

\begin{dfn}{meas}
Let $(X,.)$ be an $n$-space. Let $\Bb(X)$ denote the $\sigma$-algebra of all Borel subsets of $X$;
that is, $\Bb(X)$ is the smallest $\sigma$-algebra of subsets of $X$ which contains all open sets.
For any $\uU \in \Uu_n$ and $A \in \Bb(X)$ the set $\uU.A := \{\uU.a\dd\ a \in A\}$ is Borel
as well. We shall denote by $\chi_A\dd X \to \{0,1\}$ the characteristic function of $A$. Further,
$\Bb C^*(X,.)$ stands for the $C^*$-algebra of all bounded Borel (i.e. $\Bb(X)$-measurable)
functions $f\dd X \to M_n$ such that $f(\uU.x) = \uU.f(x)$ for any $\uU \in \Uu_n$ and
$x \in X$.\par
An \textit{$n$-measure} on $(X,.)$ is an $n \times n$-matrix $\mu = [\mu_{jk}]$ where $\mu_{jk}\dd
\Bb(X) \to \CCC$ is a regular (complex-valued) measure and $\mu(\uU.A) = \uU.\mu(A)$ for any $\uU
\in \Uu_n$ and $A \in \Bb(X)$ (here, of course, $\mu(A) = [\mu_{jk}(A)] \in M_n$). The set of all
$n$-measures on $(X,.)$ is denoted by $\MmM(X,.)$.\par
For any bounded Borel function $f\dd X \to M_n$ and an $n \times n$-matrix $\mu = [\mu_{jk}]$
of complex-valued regular Borel measures we define the integral $\int f \dint{\mu}$ as the complex
number $\sum_{j,k} \int_X f_{jk} \dint{\mu_{kj}}$ where $f(x) = [f_{jk}(x)]$ for $x \in X$.
We underline that in the formula for $\int f \dint{\mu}$, $f_{jk}$ meets $\mu_{kj}$ (not $\mu_{jk}$
[!]).
\end{dfn}

The first purpose of this section is to prove the following

\begin{thm}{dual}
For every continuous linear functional $\varphi\dd C^*(X,.) \to \CCC$ there exists a unique $\mu
\in \MmM(X,.)$ such that $\varphi(f) = \int f \dint{\mu}$ for any $f \in C^*(X,.)$.
\end{thm}

The above result is a simple consequence of the next one.

\begin{pro}{n-measure}
Let $\mu = [\mu_{jk}]$ be an $n \times n$-matrix of complex-valued regular Borel measures on $X$.
Then $\mu \in \MmM(X,.)$ iff for every map $f \in \CCc_0(X,M_n)$,
\begin{equation}\label{eqn:mu}
\int f \dint{\mu} = \int f^{\Uu} \dint{\mu}.
\end{equation}
\end{pro}
\begin{proof}
For any $n \times n$-matrix $A$ we shall write $A_{jk}$ to denote the suitable entry of $A$.
We adapt the same rule for functions $f \in \CCc_0(X,M_n)$ and matrix-valued measures. Further, for
arbitrarily fixed two indices $(j,k)$ and $(p,q)$, the function $\Uu_n \ni \uU \mapsto \uU_{jk}
\bar{\uU}_{pq} \in \CCC$ is well defined and continuous (although `$\uU_{jk}$' is not well defined).
Observe that for any $A \in M_n$, $\uU \in \Uu_n$ and an index $(p,q)$ one has:
\begin{equation*}
(\uU.A)_{p,q} = \sum_{j,k} \uU_{pj}\bar{\uU}_{qk} \cdot A_{jk} \qquad \textup{and} \qquad
(\uU^{-1}.A)_{p,q} = \sum_{j,k} \uU_{kq}\bar{\uU}_{jp} \cdot A_{jk}.
\end{equation*}
Further, for $\uU \in \Uu_n$ and a complex-valued regular Borel measure $\nu$ on $X$ let $\nu^{\uU}$
be the (complex-valued regular Borel) measure on $X$ given by $\nu^{\uU}(A) = \nu(\uU.A)\ (A \in
\Bb(X))$. It follows from the transport measure theorem that for any $g \in \CCc_0(X,\CCC)$:
$$\int_X g(\uU.x) \dint{\nu^{\uU}(x)} = \int_X g(x) \dint{\nu(x)}.$$
We adapt the above notation also for $n \times n$-matrix $\mu$ of measures: $\mu^{\uU}(A) =
\mu(\uU.A)$. Notice that $(\mu^{\uU})_{jk} = (\mu_{jk})^{\uU}$.\par
Now assume that $\mu \in \MmM(X,.)$. This means that for any $\uU \in \Uu_n$, $\uU.\mu = \mu^{\uU}$.
For $f \in \CCc_0(X,M_n)$ and $x \in X$ we have
$$(f^{\Uu})_{pq}(x) =
\sum_{j,k} \int_{\Uu_n} \uU_{kq}\bar{\uU}_{jp} \cdot f_{jk}(\uU.x) \dint{\uU}$$
and therefore, by Fubini's theorem,
\begin{multline*}
\int f^{\Uu} \dint{\mu} = \sum_{p,q} \int_X (f^{\Uu})_{p,q} \dint{\mu_{qp}} = \sum_{p,q} \sum_{j,k}
\int_X \int_{\Uu_n} \uU_{kq}\bar{\uU}_{jp} \cdot f_{jk}(\uU.x) \dint{\uU} \dint{\mu_{qp}(x)}\\
= \sum_{j,k} \int_{\Uu_n} \int_X f_{jk}(\uU.x) \dint{\Bigl(\sum_{p,q} \uU_{kq}\bar{\uU}_{jp} \cdot
\mu_{qp}\Bigr)(x)} \dint{\uU}\\= \sum_{j,k} \int_{\Uu_n} \int_X f_{jk}(\uU.x)
\dint{(\uU.\mu)_{kj}(x)} \dint{\uU} = \sum_{j,k} \int_{\Uu_n} \int_X f_{jk}(\uU.x)
\dint{(\mu_{kj})^{\uU}(x)} \dint{\uU}\\= \sum_{j,k} \int_{\Uu_n} \int_X f_{jk}(x)
\dint{\mu_{kj}(x)} \dint{\uU} = \sum_{j,k} \int_X f_{jk}(x) \dint{\mu_{kj}(x)} = \int f \dint{\mu},
\end{multline*}
which gives \eqref{eqn:mu}. Conversely, assume \eqref{eqn:mu} if fulfilled for any $f \in
\CCc_0(X,M_n)$ and fix a compact $\GGg_{\delta}$ subset $K$ of $X$ and an index $(p,q)$. Let $g \in
\CCc_0(X,\CCC)$ be arbitrary and let $f \in \CCc_0(X,M_n)$ be such that $f_{pq} = g$ and $f_{jk} =
0$ for $(j,k) \neq (p,q)$. Applying \eqref{eqn:mu} for such $f$, we obtain
\begin{equation}\label{eqn:31}
\int_X g \dint{\mu_{qp}} = \sum_{j,k} \int_X \int_{\Uu_n} \uU_{qk}\bar{\uU}_{pj} \cdot g(\uU.x)
\dint{\uU} \dint{\mu_{kj}(x)}.
\end{equation}
Further, since $K$ is compact and $\GGg_{\delta}$, there is a sequence $(g_k)_{k=1}^{\infty} \subset
\CCc_0(X,\CCC)$ such that $g_k(X) \subset [0,1]$ and $\lim_{k\to\infty} g_k(x) = \chi_K(x)$ for any
$x \in X$. Substituting $g = g_k$ in \eqref{eqn:31} and letting $k \to \infty$, we obtain
(by Lebesgue's dominated convergence theorem and Fubini's one):
\begin{multline*}
\mu_{q,p}(K) = \sum_{j,k} \int_{\Uu_n} \int_X \uU_{qk}\bar{\uU}_{pj} \cdot \chi_K(\uU.x)
\dint{\mu_{kj}(x)} \dint{\uU}\\= \int_{\Uu_n} \Bigl(\sum_{j,k} \uU_{qk}\bar{\uU}_{pj} \cdot
\mu_{kj}(\uU^{-1}.K)\Bigr) \dint{\uU} = \int_{\Uu_n} (\uU.\mu)_{q,p}(\uU^{-1}.K) \dint{\uU}.
\end{multline*}
We infer from the arbitrarity of $(p,q)$ in the above formula that
$$\mu(K) = \int_{\Uu_n} \uU.\mu(\uU^{-1}.K) \dint{\uU}.$$
Now if $\vV \in \Uu_n$, the set $\vV.K$ is also compact and $\GGg_{\delta}$ and therefore
\begin{multline*}
\mu(\vV.K) = \int_{\Uu_n} \uU.\mu(\uU^{-1}\vV.K) \dint{\uU} = \int_{\Uu_n} \vV.[\uU.\mu(\uU^{-1}.K)]
\dint{\uU}\\= \vV.\Bigl(\int_{\Uu_n} \uU.\mu(\uU^{-1}.K) \dint{\uU}\Bigr) = \vV.\mu(K).
\end{multline*}
Finally, since $\mu$ is regular, the relation $\mu(\vV.A) = \vV.\mu(A)$ holds for any $A \in \Bb(X)$
and we are done.
\end{proof}

\begin{proof}[Proof of \THM{dual}]
Note that the function $P\dd \CCc_0(X,M_n) \ni f \mapsto f^{\Uu} \in C^*(X,.)$ is a continuous
linear projection (that is, $P(f) = f$ for $f \in C^*(X,.)$). So, if $\varphi\dd C^*(X,.) \to
\CCC$ is a continuous linear functional, so is $\psi := \varphi \circ P\dd \CCc_0(X,M_n) \to \CCC$.
Since $\CCc_0(X,M_n)$ is isomorphic, as a Banach space, to $[\CCc_0(X,\CCC)]^{n^2}$, the Riesz-type
representation theorem yields that there is a unique $n \times n$-matrix $\mu$ of complex-valued
regular Borel measures such that $\psi(f) = \int f \dint{\mu}$. Observe that $\psi(f^{\Uu}) =
\psi(f)$ for any $f \in \CCc_0(X,M_n)$ and hence $\mu \in \MmM(X,.)$, thanks to \PRO{n-measure}.
The uniqueness of $\mu$ follows from the above construction, \PRO{n-measure} and the uniqueness
in the Riesz-type representation theorem.
\end{proof}

Now we turn to representations of $C^*(X,.)$. To this end, we introduce

\begin{dfn}{n-measure}
An operator-valued \textit{$n$-measure} on the $n$-space $(X,.)$ is any function of the form $E\dd
\Bb(X) \ni A \mapsto [E_{jk}(A)] \in M_n(\BBb(\HHh))$ (where $(\HHh,\scalarr)$ is a Hilbert space)
such that:
\begin{enumerate}[(M1)]
\item for any $h, w \in \HHh$ and $j,k \in \{1,\ldots,n\}$, the function $E_{jk}^{(h,w)}\dd \Bb(X)
   \ni A \mapsto \scalar{E_{jk}(A) h}{w} \in \CCC$ is a (complex-valued) measure,
\item for any $\uU \in \Uu_n$ and $A \in \Bb(X)$, $E(\uU.A) = \uU.E(A)$.
\end{enumerate}
In other words, an operator-valued $n$-measure is an $n \times n$-matrix of operator-valued measures
which satisfies axiom (M2). The operator-valued $n$-measure $E$ is \textit{regular} iff
$E_{jk}^{(h,w)}$ is regular for any $h,w$ and $j,k$.
\end{dfn}

Recall that if $\mu\dd \Bb(X) \to \BBb(\HHh)$ is an operator-valued measure and $f\dd X \to \CCC$ is
a bounded Borel function, $\int_X f \dint{\mu}$ is a bounded linear operator on $\HHh$, defined
by an implicit formula:
$$\Bigl\langle\Bigl(\int_X f \dint{\mu}\Bigr)h,w\Bigr\rangle = \int_X f \dint{\mu^{(h,w)}} \qquad
(h,w \in \HHh),$$
where $\mu^{(h,w)}(A) = \scalar{\mu(A)h}{w}\ (A \in \Bb(X))$. Now assume that $E = [E_{jk}]\dd
\Bb(X) \to M_n(\BBb(\HHh))$ is an $n$-measure and $f = [f_{jk}]\dd X \to M_n$ is a bounded Borel
function. We define $\int f \dint{E}$ as a bounded linear operator on $\HHh$ given by $\int f
\dint{E} = \sum_{j,k} \int_X f_{jk} \dint{E_{kj}}$. We are now ready to introduce

\begin{dfn}{spectral-n}
A \textit{spectral} $n$-measure is any operator-valued regular $n$-measure $E\dd \Bb(X) \to
M_n(\BBb(\HHh))$ such that
\begin{gather}
\Bigl(\int f \dint{E}\Bigr)^* = \int f^* \dint{E},\label{eqn:adjoint}\\
\int f \cdot g \dint{E} = \int f \dint{E} \cdot \int g \dint{E}\label{eqn:multi}
\end{gather}
for any $f, g \in \Bb C^*(X,.)$. (The product $f \cdot g$ is computed pointwise as the product
of matrices.) In other words, a spectral $n$-measure is an operator-valued regular $n$-measure $E\dd
\Bb(X) \to M_n(\BBb(\HHh))$ such that the operator
$$\Bb C^*(X,.) \ni f \mapsto \int f \dint{E} \in \BBb(\HHh)$$
is a representation of a $C^*$-algebra $\Bb C^*(X,.)$.
\end{dfn}

The main result of this section is the following

\begin{thm}{represent}
Let $(X,.)$ be an $n$-space and $\pi\dd C^*(X,.) \to \BBb(\HHh)$ be a representation. There is
a unique spectral $n$-measure $E\dd \Bb(X) \to M_n(\BBb(\HHh))$ such that
\begin{equation}\label{eqn:represent}
\pi(f) = \int f \dint{E} \qquad (f \in C^*(X,.)).
\end{equation}
In particular, every representation of $C^*(X,.)$ admits an extension to a representation
of $\Bb C^*(X,.)$.
\end{thm}

In the proof of the above result we shall involve the following

\begin{lem}{dense}
Let $\mu\dd \Bb(X) \to \RRR_+$ be a regular measure. For any $f \in \Bb C^*(X,.)$ and $\epsi > 0$
there exists $g \in C^*(X,.)$ such that $\sup_{x \in X} \|g(x)\| \leqsl \sup_{x \in X} \|f(x)\|$
and $\int_X \|f(x) - g(x)\| \dint{\mu(x)} < \epsi$.
\end{lem}
\begin{proof}
Let $f = [f_{jk}] \in \Bb C^*(X,.)$ and let $M > 0$ be such that
$$\sup_{x \in X} \|f(x)\| \leqsl M.$$
It follows from the regularity of $\mu$ that for each $(j,k)$ there is a compact set $L_{jk}$ such
that $\mu(X \setminus L_{jk}) \leqsl \frac{\epsi}{2 M n^2}$ and $f_{jk}\bigr|_{L_{jk}}$ is
continuous. Put $L = \bigcap_{j,k} L_{jk}$ and $K = \Uu_n.L$. Then $K$ is compact and invariant, and
$\mu(X \setminus K) \leqsl \frac{\epsi}{2M}$. What is more, $f\bigr|_K$ is continuous (this follows
from the facts that $f\bigr|_L$ is continuous and $f(\uU.x) = \uU.f(x)$). Now point (c) of \LEM{ext}
yields the existence of $g \in C^*(X,.)$ such that $\sup_{x \in X} \|g(x)\| \leqsl \sup_{x \in X}
\|f(x)\|$ and $g\bigr|_K = f\bigr|_K$. Then:
$$\int_X \|f(x) - g(x)\| \dint{\mu(x)} = \int_{X \setminus K} \|f(x) - g(x)\| \dint{\mu} \leqsl
2M \cdot \mu(X \setminus K) = \epsi$$
and we are done.
\end{proof}

\begin{pro}{n-spectral}
Let $E = [E_{jk}]\dd \Bb(X) \to M_n(\BBb(\HHh))$ be a regular $n$-measure.
\begin{enumerate}[\upshape(a)]
\item $E$ satisfies \eqref{eqn:adjoint} for any $f \in \Bb C^*(X,.)$ iff \eqref{eqn:adjoint} is
   fulfilled for any $f \in C^*(X,.)$, iff $(E_{jk}(A))^* = E_{kj}(A)$ for each $A \in \Bb(X)$,
\item $E$ is spectral iff \eqref{eqn:adjoint} and \eqref{eqn:multi} are satisfied for any $f, g \in
   C^*(X,.)$.
\end{enumerate}
\end{pro}
\begin{proof}
For any complex-valued regular Borel measure $\nu$ on $X$ we shall denote by $|\nu|$ the variation
of $\nu$. Recall that $|\nu|$ is a nonnegative finite regular Borel measure on $X$. Further, for any
$h,w \in \HHh$ and $j,k \in \{1,\ldots,n\}$ let $E_{jk}^{(h,w)}$ be as in \DEF{n-measure}. Finally,
$\scalarr$ stands for the scalar product of $\HHh$.\par
We begin with (a). Fix $h, w \in \HHh$ and $j, k \in \{1,\ldots,n\}$. First assume that
\eqref{eqn:adjoint} is fulfilled for any $f \in C^*(X,.)$. Let $E^{(h,w)} := [E_{jk}^{(h,w)}]$
and note that $E^{(h,w)} \in \MmM(X,.)$, since $E_{pq}(\uU.A) = \sum_{j,k} \uU_{pj}\bar{\uU}_{qk}
\cdot E_{jk}(A)$ and thus $E_{pq}^{(h,w)}(\uU.A) = \sum_{j,k} \uU_{pj}\bar{\uU}_{qk} \cdot
E_{jk}^{(h,w)}(A) = (\uU.E^{(h,w)}(A))_{pq}.$ Observe that $(E^{(h,w)})^* \in \MmM(X,.)$ as well
where $(E^{(h,w)})^*(A) = (E^{(h,w)}(A))^*$ (because $(\uU.P)^* = \uU.P^*$ for any $P \in M_n$).
Further, for each $f \in C^*(X,.)$ we have
\begin{multline*}
\overline{\int f^* \dint{E^{(h,w)}}} = \sum_{j,k} \overline{\int_X (f^*)_{jk} \dint{E_{kj}^{(h,w)}}}
= \sum_{j,k} \int_X f_{kj} \dint{\overline{E_{kj}^{(h,w)}}}\\
= \sum_{j,k} \int_X f_{kj} \dint{(E^{(h,w)})^*_{jk}} = \int f \dint{(E^{(h,w)})^*}
\end{multline*}
and, on the other hand,
\begin{multline*}
\overline{\int f^* \dint{E^{(h,w)}}} = \overline{\left\langle\Bigl(\int f^*
\dint{E}\Bigr)h,w\right\rangle} =
\overline{\left\langle\Bigl(\int f \dint{E}\Bigr)^* h,w\right\rangle}\\
= \left\langle\Bigl(\int f \dint{E}\Bigr)w,h\right\rangle = \int f \dint{E^{(w,h)}}.
\end{multline*}
The uniqueness in \THM{dual} implies that $(E^{(h,w)})^* = E^{(w,h)}$, which means that for each
$A \in \Bb(X)$, $\scalar{(E_{jk}(A))w}{h} = \overline{\scalar{(E_{kj}(A))h}{w}} =
\scalar{(E_{kj}(A))^* w}{h}$. We conclude that $(E_{jk}(A))^* = E_{kj}(A)$. Finally, if the last
relation holds for any $j,k \in \{1,\ldots,n\}$, then for every $f \in \Bb C^*(X,.)$ we get
\begin{multline*}
\Bigl(\int f \dint{E}\Bigr)^* = \sum_{j,k} \Bigl(\int_X f_{jk} \dint{E_{kj}}\Bigr)^* =
\sum_{j,k} \int_X \bar{f}_{jk} \dint{(E_{kj})^*}\\= \sum_{j,k} \int_X (f^*)_{kj} \dint{E_{jk}} =
\int f^* \dint{E}.
\end{multline*}
This completes the proof of (a). We now turn to (b). We assume that \eqref{eqn:adjoint} and
\eqref{eqn:multi} are fulfilled for any $f,g \in C^*(X,.)$. We know from (a) that actually
\eqref{eqn:adjoint} is satisfied for any $f \in \Bb C^*(X,.)$. The proof of \eqref{eqn:multi} is
divided into three steps stated below.\par
\underline{Step 1}. If $\xi \in \Bb C^*(X,.)$ is such that
\begin{equation}\label{eqn:41}
\int g \cdot \xi \dint{E} = \int g \dint{E} \cdot \int \xi \dint{E}
\end{equation}
for any $g \in C^*(X,.)$, then $\int f \cdot \xi \dint{E} = \int f \dint{E} \cdot \int \xi
\dint{E}$ for any $f \in \Bb C^*(X,.)$.\par
Proof of step 1. Fix $f \in \Bb C^*(X,.)$, $h, w \in \HHh$ and $\epsi > 0$. Let $M \geqsl 1$ be
such that $\sup_{x \in X} \|\xi(x)\| \leqsl M$. Put $v = (\int \xi \dint{E}) h$ and $\mu =
\sum_{j,k} (|E_{jk}^{(h,w)}| + |E_{jk}^{(v,w)}|)$. Since $\mu$ is finite and regular, \LEM{dense}
gives us a map $g \in C^*(X,.)$ such that $\int_X \|f(x) - g(x)\| \dint{\mu(x)} \leqsl
\frac{\epsi}{M}$. Then \eqref{eqn:41} holds and therefore (remember that $M \geqsl 1$):
\begin{multline*}
\left|\left\langle\Bigl(\int f \cdot \xi \dint{E} - \int f \dint{E} \cdot \int \xi
\dint{E}\Bigr)h,w\right\rangle\right| \leqsl \left|\left\langle\Bigl(\int f \cdot \xi \dint{E}
- \int g \cdot \xi \dint{E}\Bigr)h,w\right\rangle\right|\\+ \left|\left\langle\Bigl(\int g \dint{E}
\cdot \int \xi \dint{E} - \int f \dint{E} \cdot \int \xi \dint{E}\Bigr)h,w\right\rangle\right|\\
= \Bigl|\sum_{j,k} \int_X ((f - g)\xi)_{jk} \dint{E_{kj}^{(h,w)}}\Bigr| + \Bigl|\sum_{j,k} \int_X
(g_{jk} - f_{jk}) \dint{E_{kj}^{(v,w)}}\Bigr|\\\leqsl \sum_{j,k} \int_X \|(f(x) - g(x))\xi(x)\|
\dint{|E_{kj}^{(h,w)}|(x)} + \sum_{j,k} \int_X \|g(x) - f(x)\| \dint{|E_{kj}^{(v,w)}|(x)}\\\leqsl M
\int_X \|f(x) - g(x)\| \dint{\mu(x)} \leqsl \epsi.
\end{multline*}\par
\underline{Step 2}. For any $f \in \Bb C^*(X,.)$ and $g \in \CCc(X,.)$, \eqref{eqn:multi}
holds.\par
Proof of step 2. It follows from step 1 and our assumptions in (b) that $\int g^* \cdot f^* \dint{E}
= \int g^* \dint{E} \cdot \int f^* \dint{E}$. Now it suffices to apply \eqref{eqn:adjoint}:
$$\int f \cdot g \dint{E} = \Bigl(\int g^* \cdot f^* \dint{E}\Bigr)^* = \Bigl(\int g^* \dint{E}
\cdot \int f^* \dint{E}\Bigr)^* = \int f \dint{E} \cdot \int g \dint{E}.$$\par
\underline{Step 3}. The condition \eqref{eqn:multi} is satisfied for any $f, g \in
\Bb C^*(X,.)$.\par
Proof of step 3. Just apply step 2 and then step 1.
\end{proof}

\begin{proof}[Proof of \THM{represent}]
According to point (b) of \PRO{n-spectral}, it suffices to show that there exists a regular
$n$-measure $E\dd \Bb(X) \to M_n(\BBb(\HHh))$ such that \eqref{eqn:represent} holds and that such
$E$ is unique. According to \THM{dual}, for any $h, w \in \HHh$ there is a unique $\mu^{(h,w)} =
[\mu^{(h,w)}_{jk}] \in \MmM(X,.)$ such that
\begin{equation}\label{eqn:42}
\scalar{\pi(f)h}{w} = \int f \dint{\mu^{(h,w)}}
\end{equation}
for each $f \in C^*(X,.)$ ($\scalarr$ is the scalar product of $\HHh$). Now for any $j,k \in
\{1,\ldots,n\}$ and each $A \in \Bb(X)$ there is a unique bounded operator on $\HHh$, denoted
by $E_{jk}(A)$, for which $\mu^{(h,w)}_{jk}(A) = \scalar{(E_{jk}(A))h}{w}\ (h,w \in \HHh)$. We put
$E(A) = [E_{jk}(A)] \in M_n(\BBb(\HHh))$. We want to show that $E(\uU.A) = \uU.E(A)$. Since
$\mu^{(h,w)} \in \MmM(X,.)$, we obtain
\begin{multline*}
\scalar{(E_{pq}(\uU.A))h}{w} = (\mu^{(h,w)}(\uU.A))_{pq} = (\uU.\mu^{(h,w)}(A))_{pq} =
\sum_{j,k} \uU_{pj}\bar{\uU}_{qk} \cdot \mu^{(h,w)}_{jk}(A)\\
= \sum_{j,k} \uU_{pj}\bar{\uU}_{qk} \cdot \scalar{(E_{jk}(A))h}{w} = \scalar{(\uU.E(A))_{pq} h}{w},
\end{multline*}
which shows that indeed $E(\uU.A) = \uU.E(A)$. Further, observe that $E_{jk}^{(h,w)} =
\mu^{(h,w)}_{jk}$ and thus $E$ is an operator-valued regular $n$-measure and
$\scalar{(\int f \dint{E})h}{w} = \scalar{\pi(f)h}{w}$ (thanks to \eqref{eqn:42}). Consequently,
$\int f \dint{E} = \pi(f)$ and we are done.\par
The uniqueness of $E$ follows from the above construction and its proof is left to the reader.
\end{proof}

\begin{exm}{spectral-n}
Let $(X,.)$ be an $n$-space and let $E = [E_{jk}]\dd \Bb(X) \to M_n(\BBb(\HHh))$ be a spectral
$n$-measure. We denote by $\Bb_{inv}(X)$ the $\sigma$-algebra of all invariant Borel subsets of $X$
(that is, $A \in \Bb(X)$ belongs to $\Bb_{inv}(X)$ iff $\uU.A = A$ for any $\uU \in \Uu_n$). Let
$F\dd \Bb_{inv}(X) \ni A \mapsto \sum_j E_{jj}(A) \in \BBb(\HHh)$. Then for every $A \in
\Bb_{inv}(X)$ one has:
\begin{enumerate}[(E1)]
\item $E_{jk}(A) = 0$ whenever $j \neq k$,
\item $E_{11}(A) = \ldots = E_{nn}(A) = \frac1n F(A)$
\end{enumerate}
and $F$ is a spectral measure (possibly with $F(X) \neq I_{\HHh}$ where $I_{\HHh}$ is the identity
operator on $\HHh$). Let us briefly prove these claims. Since $E(A) = E(\uU.A) = \uU.E(A)$ for any
$\uU \in \Uu_n$, conditions (E1)--(E2) are fulfilled. Further, if $j_A\dd X \to M_n$ is given
by $j_A(x) = \chi_A(x) \cdot I$ where $I \in M_n$ is the unit matrix, then $j_A \in \Bb C^*(X,.)$
and for $B \in \Bb_{inv}(X)$, $F(A \cap B) = \int j_{A \cap B} \dint{E} = \int j_A \cdot j_B
\dint{E} = \int j_A \dint{E} \cdot \int j_B \dint{E} = F(A) F(B)$. What is more, point (a)
of \PRO{n-spectral} implies that $F(A)$ is selfadjoint and hence $F$ is indeed a spectral measure.
One may also easily check that $F(X) = I_{\HHh}$ iff the representation $\pi_E\dd C^*(X,.) \ni f
\mapsto \int f \dint{E} \in \BBb(\HHh)$ is nondegenerate.\par
The spectral measure $F$ defined above corresponds to the representation of the center $Z$
of $C^*(X,.)$. It is a simple exercise that $Z$ consists precisely of all $f \in \CCc_0(X,\CCC
\cdot I)$ which are constant on the sets of the form $\Uu_n.\{x\}\ (x \in X)$. Thus, $\Bb_{inv}(X)$
may naturally be identified with the Borel $\sigma$-algebra of the spectrum of $Z$ and consequently
$F$ is the spectral measure induced by the representation $\pi_E\bigr|_Z$ of $Z$.\par
Conditions (E1)--(E2) show that a nonzero spectral $n$-measure $E$ for $n > 1$ \textit{never}
satisfies the condition of a spectral measure --- that $E(A \cap B) = E(A) E(B)$. Indeed, $E(X) \neq
(E(X))^2$.
\end{exm}

The next result is well known. For reader convenience, we give its short proof.

\begin{lem}{irr}
Let $\AAa$ be a $C^*$-algebra and let $\pi\dd \AAa \to M_n$ \textup{(}where $n \geqsl 1$ is
finite\textup{)} be a nonzero irreducible representation of $\AAa$. Then $\pi$ is surjective.
\end{lem}
\begin{proof}
Let $J = \pi(\AAa)$. Since $\pi$ is irreducible, $J' = \CCC \cdot I$ and consequently $J'' = M_n$.
But it follows from von Neumann's double commutant theorem that $J'' = J + \CCC \cdot I$ (here
we use the fact that $n$ is finite). So, the facts that $J$ is a $*$-algebra and $M_n = J + \CCC
\cdot I$ imply that $J$ is a two-sided ideal in $M_n$. Consequently, $J = \{0\}$ or $J = M_n$. But
$\pi \neq 0$ and hence $J = M_n$.
\end{proof}

With the aid of the above lemma and \THM{represent} we shall now characterize all irreducible
representations of $C^*(X,.)$.

\begin{pro}{irr}
Every nonzero irreducible representation $\pi$ of $C^*(X,.)$ \textup{(}where $(X,.)$ is
an $n$-space\textup{)} is $n$-dimensional and has the form $\pi(f) = f(x)$ \textup{(}for some $x \in
X$\textup{)}.
\end{pro}
\begin{proof}
Let $\pi\dd C^*(X,.) \to \BBb(\HHh)$ be a nonzero irreducible representation. It follows from
\THM{represent} that there is a spectral $n$-measure $E\dd \Bb(X) \to \BBb(\HHh)$ such that
\eqref{eqn:represent} holds. Let $\Bb_{inv}(X)$ and $F\dd \Bb_{inv}(X) \to \BBb(\HHh)$ be
as in \EXM{spectral-n}. Then $F$ is a (regular) spectral measure (with $F(X) = I_{\HHh}$ because
$\pi$ is nondegenerate) and for any $A \in \Bb_{inv}(X)$ and $f \in C^*(X,.)$ we have $\int f
\dint{E} \cdot \int \chi_A I \dint{E} = \int \chi_A I \dint{E} \cdot \int f \dint{E}$ (where $I$ is
the unit $n \times n$-matrix). Since $\pi$ is irreducible, we deduce that for every $A \in
\Bb_{inv}(X)$, $\int \chi_A I \dint{E}$ is a scalar multiple of the identity operator on $\HHh$.
This implies that we may think of $F$ as a complex-valued (spectral) measure. But $\Bb_{inv}(X)$ is
naturally `isomorphic' to the $\sigma$-algebra of all Borel sets of $X/\Uu_n$ (which is locally
compact) and thus $F$ is supported on a set $S := \Uu_n.a$ for some $a \in X$. But then $\int \chi_X
I \dint{E} = \int \chi_S I \dint{E}$ and consequently $\pi(f) = \int f\bigr|_S \dint{E_S}$ where
$E_S$ is the restriction of $E$ to $\Bb(S)$. Since the vector space $\{f\bigr|_S\dd\ f \in
C^*(X,.)\}$ is finite dimensional (and its dimension is equal to $n^2$), we infer that $\AAa :=
\pi(C^*(X,.))$ is finite dimensional as well and $\dim \AAa \leqsl n^2$. So, the irreducibility
of $\pi$ implies that $\HHh$ is finite dimensional, while \LEM{irr} shows that $\dim \HHh \leqsl n$.
Finally, point (b) of \PRO{repr} completes the proof.
\end{proof}

\SECT{Homogeneous $C^*$-algebras}

\begin{dfn}{pure}
A $C^*$-algebra is said to be \textit{$n$-homogeneous} (where $n$ is finite) iff every its nonzero
irreducible representation is $n$-dimensional.
\end{dfn}

Our version of Fell's characterization of $n$\hyp{}homogeneous $C^*$-algebras \cite{fe3} reads
as follows.

\begin{thm}{homog}
For a $C^*$-algebra $\AAa$ and finite $n \geqsl 1$ \tfcae
\begin{enumerate}[\upshape(i)]
\item $\AAa$ is an $n$-homogeneous $C^*$-algebra;
\item there is an $n$-space $(X,.)$ such that $\AAa$ is isomorphic
   \textup{(}as a $C^*$-algebra\textup{)} to $C^*(X,.)$.
\end{enumerate}
What is more, if $\AAa$ is $n$-homogeneous, the $n$-space $(X,.)$ appearing in \textup{(ii)} is
unique up to isomorphism.
\end{thm}
\begin{proof}[Proof of \THM{homog}]
We infer from point (c) of \PRO{repr} that the $n$-space $(X,.)$ appearing in (ii) is unique
up to isomorphism. What is more, it easily follows from \PRO{irr} that $C^*(X,.)$ is
$n$-homogeneous for any $n$-space $(X,.)$. So, it remains to show that (i) implies (ii). To this
end, assume $\AAa$ is $n$-homogeneous and let $\Xx$ be the set of all representations (including
the zero one) $\pi\dd \AAa \to M_n$, equipped with the topology of pointwise convergence. Since each
representation is a bounded linear operator of norm not greater than $1$, $\Xx$ is compact.
Consequently, $X := \Xx \setminus \{0\}$ is locally compact. We define an action of $\Uu_n$ on $X$
by the formula: $(\uU.\pi)(a) = \uU.\pi(a)\ (a \in \AAa,\ \pi \in X,\ \uU \in \Uu_n)$. It is easily
seen that the action is continuous. What is more, \LEM{irr} ensures us that it is free as well. So,
$(X,.)$ is an $n$-space. The next step of construction is very common. For any $a \in \AAa$ let
$\widehat{a}\dd X \to M_n$ be given by $\widehat{a}(\pi) = \pi(a)$. It is clear that $\widehat{a}
\in \CCc_0(X,M_n)$ (indeed, if $X$ is noncompact, then $\Xx = X \cup \{0\}$ is a one-point
compactification of $X$ and $\widehat{a}$ extends to a map on $\Xx$ which vanishes at $0$). We also
readily have $\widehat{a}(\uU.\pi) = \uU.\widehat{a}(\pi)$ for any $\uU \in \Uu_n$. So, we have
obtained a $*$-homomorphism $\Phi\dd \AAa \ni a \mapsto \widehat{a} \in C^*(X,.)$. It follows from
(i) (and the fact that all irreducible representations separate points of a $C^*$-algebra) that
$\Phi$ is one-to-one and, consequently, $\Phi$ is isometric. So, to end the proof, it suffices
to show that $\EeE = \Phi(\AAa)$ is dense in $C^*(X,.)$. To this end, we involve \THM{SW-full}.
It follows from \LEM{irr} that condition (AX0) is fulfilled. Further, let $\pi_1$ and $\pi_2$ be
arbitrary members of $X$. We consider two cases. First assume that $\pi_2 = \uU.\pi_1$ for some $\uU
\in \Uu_n$. Then $\widehat{a}(\pi_2) = \uU.\widehat{a}(\pi_1)$ and consequently
$\sigma(\widehat{a}(\pi_1)) = \sigma(\widehat{a}(\pi_2))\ (a \in \AAa)$. So, in that case (AX2)
holds. Now assume that there is no $\uU \in \Uu_n$ for which $\pi_2 = \uU.\pi_1$. We shall show that
in that case:
\begin{equation}\label{eqn:21}
\pi_1(a) = 0 \qquad \textup{and} \quad \pi_2(a) = I \qquad \textup{for some } a \in \AAa.
\end{equation}
Let $\MMm \subset M_{2n}$ consist of all matrices of the form
$\begin{pmatrix}\pi_1(x) & 0\\0 & \pi_2(x)\end{pmatrix}$ with $x \in \AAa$. Since $\MMm$ is
a finite-dimensional $C^*$-algebra, it is singly generated (see e.g. \cite{sai}) and unital (cf.
\cite[\S1.11]{tak}). Thanks to \LEM{irr}, $\MMm$ contains matrices of the form
$\begin{pmatrix}I & 0\\0 & A\end{pmatrix}$ and $\begin{pmatrix}B & 0\\0 & I\end{pmatrix}$ (for some
$A, B \in M_n$). We conclude that the unit of $\MMm$ coincides with the unit of $M_{2n}$. This,
combined with the fact that $\MMm$ is singly generated, yields that there is $z \in \AAa$ such that
for $A_j = \pi_j(z)\ (j=1,2)$ we have
$$\MMm = \left\{\begin{pmatrix}p(A_1,A_1^*) & 0\\0 & p(A_2,A_2^*)\end{pmatrix}\dd\
p \in \PPp\right\}$$
where $\PPp$ is the free algebra of all polynomials in two noncommuting variables. Observe that then
$M_n = \pi_j(\AAa) = \{p(A_j,A_j^*)\dd\ p \in \PPp\}\ (j=1,2)$, which means that $A_1$ and $A_2$ are
irreducible matrices. What is more, $A_1$ and $A_2$ are not unitarily equivalent, that is, there is
no $\uU \in \Uu_n$ for which $A_2 = \uU.A_1$ (indeed, if $A_2 = \uU.A_1$, then $\pi_2 = \uU.\pi_1$,
since for every $x \in \AAa$ there is $p \in \PPp$ such that $\pi_j(x) = p(A_j,A_j^*)$). These two
remarks imply that $\begin{pmatrix}0 & 0\\ 0 & I\end{pmatrix} \in \MMm$, because the $*$-commutant
in $M_{2n}$ of the matrix $\begin{pmatrix}A_1 & 0\\0 & A_2\end{pmatrix}$ consists of matrices
of the form $\begin{pmatrix}X & 0\\0 & Y\end{pmatrix}$ (this follows from the so-called Schur's
lemma on intertwining transformations---see Theorem~1.5 in \cite{ern} and Corollary~1.8 there; cf.
also Proposition~5.2.1 in \cite{pn0}) and consequently $\begin{pmatrix}0 & 0\\ 0 & I\end{pmatrix}
\in \MMm'' = \MMm$. So, there is $a \in \AAa$ such that $\begin{pmatrix}0 & 0\\ 0 & I\end{pmatrix} =
\begin{pmatrix}\pi_1(a) & 0\\0 & \pi_2(a)\end{pmatrix}$, which gives \eqref{eqn:21}. Replacing $a$
by $\frac{a + a^*}{2}$, we may assume that $a$ is selfadjoint. Then $f = \widehat{a} \in \EeE$ is
selfadjoint (and hence normal) and $\sigma(f(\pi_1)) \cap \sigma(f(\pi_2)) = \varempty$, which shows
that $\pi_1$ and $\pi_2$ are spectrally separated by $\EeE$. According to \THM{SW-full},
it therefore suffices to check that each $g \in C^*(X,.)$ belongs to $\Delta_2(\EeE)$. To this
end, we fix $\pi_1, \pi_2 \in X$ and consider the same two cases as before. If $\pi_2 = \uU.\pi_1$,
it follows from \LEM{irr} that there is $x \in \AAa$ for which $\pi_1(x) = g(\pi_1)$. Then
$\widehat{x}(\pi_1) = g(\pi_1)$ and $\widehat{x}(\pi_2) = \uU.\widehat{x}(\pi_1) = \uU.g(\pi_1) =
g(\pi_2)$ and we are done.\par
Finally, if $\pi_2 \neq \uU.\pi_1$ for any $\uU \in \Uu_n$, \eqref{eqn:21} implies that there are
points $a_1$ and $a_2$ in $\AAa$ such that $\pi_1(a_1) = I = \pi_2(a_2)$ and $\pi_1(a_2) = 0 =
\pi_2(a_1)$. Moreover, there are points $x, y \in \AAa$ such that $\pi_1(x) = g(\pi_1)$ and
$\pi_2(y) = g(\pi_2)$ (by \LEM{irr}). Put $z = x a_1 + y a_2 \in \AAa$ and note that
$\widehat{z}(\pi_j) = g(\pi_j)$ for $j=1,2$, which means that $g \in \Delta_2(\EeE)$. The whole
proof is complete.
\end{proof}

\begin{dfn}{n-spectrum}
Let $\AAa$ be an $n$-homogeneous $C^*$-algebra. By an \textit{$n$-spectrum} of $\AAa$ we mean
any $n$-space $(X,.)$ such that $\AAa$ is isomorphic to $C^*(X,.)$. It follows from \THM{homog} that
an $n$-spectrum of $\AAa$ is unique up to isomorphism of $n$-spaces. By \textit{concrete
$n$-spectrum} of $\AAa$ we mean the $n$-space of all nonzero representations $\pi\dd \AAa \to M_n$
endowed with the pointwise convergence topology and the natural action of $\Uu_n$.\par
The trivial algebra $\{0\}$ is $n$-homogeneous and its $n$-spectrum is the empty $n$-space.
\end{dfn}

The reader interested in general ideas of operator spectra should consult \cite{fe1,fe2,fe3};
\cite{chi} as well as \cite{ha1,ha2}; \cite{p-s}; \S2.5 in \cite{ern}; \cite{kkl} or \cite{lee}.\par
Our approach to $n$-homogeneous $C^*$-algebras allows us to prove briefly the following

\begin{pro}{ext}
Let $\AAa_1$ and $\AAa_2$ be two $n$-homogeneous $C^*$-algebras such that $\AAa_1 \subset \AAa_2$.
\begin{enumerate}[\upshape(a)]
\item Every representation $\pi_1\dd \AAa_1 \to M_n$ is extendable to a representation $\pi_2\dd
   \AAa_2 \to M_n$.
\item If every $n$-dimensional representation \textup{(}including the zero one\textup{)} of $\AAa_1$
   has a unique extension to an $n$-dimensional representation of $\AAa_2$, then $\AAa_1 = \AAa_2$.
\end{enumerate}
\end{pro}
\begin{proof}
We begin with (a). We may and do assume that $\pi_1$ is nonzero. For $j=1,2$ let $(X_j,.)$ denote
an $n$-spectrum of $\AAa_j$ and let $\Psi_j\dd \AAa_j \to C^*(X_j,.)$ be a $*$-isomorphism
of $C^*$-algebras. Let $j\dd \AAa_1 \to \AAa_2$ be the inclusion map. Then $\Phi := \Psi_2 \circ
j \circ \Psi_1^{-1}\dd C^*(X_1,.) \to C^*(X_2,.)$ is a one-to-one $*$-homomorphism. We infer
from \PRO{repr} that there are an invariant open (in $X_2$) set $U$ and a morphism $\varphi\dd (U,.)
\to (X_1,.)$ such that \eqref{eqn:homo} holds. We claim that
\begin{equation}\label{eqn:22}
\varphi(U) = X_1.
\end{equation}
Since $\varphi$ is proper, the set $F := \varphi(U)$ is closed in $X_1$. It is also invariant. So,
if $F \neq X_1$, we may take $b \in X_1 \setminus F$ and apply point (c) of \LEM{ext} to get
a function $f \in C^*(X_1,.)$ such that $f\bigr|_F \equiv 0$ and $f(b) = I$. Then $\Phi(f) = 0$,
by \eqref{eqn:homo}, which contradicts the fact that $\Phi$ is one-to-one. So, \eqref{eqn:22} is
fulfilled.\par
Further, \PRO{repr} yields that there is $x \in X_1$ such that $\pi_1(\Psi_1^{-1}(f)) = f(x)$ for
any $f \in C^*(X_1,.)$. It follows from \eqref{eqn:22} that we may find $z \in U$ for which
$\varphi(z) = x$. Now define $\pi_2\dd \AAa_2 \to M_n$ by $\pi_2(a) = [\Psi_2(a)](z)\ (a \in
\AAa_2)$. It remains to check that $\pi_2$ extends $\pi_1$. To see this, for $a \in \AAa_1$ put $f =
\Psi_1(a)$ and note that $\pi_2(a) = [\Psi_2(a)](z) = [\Psi_2(\Psi_1^{-1}(f))](z) = [\Phi(f)](z) =
f(\varphi(z)) = f(x) = \pi_1(a)$ (cf. \eqref{eqn:homo}).\par
Now if the assumption of (b) is satisfied, the above argument shows that $\varphi$ is one-to-one
(since different points of $X_2$ correspond to different $n$-dimensional representations
of $\AAa_2$). It may also easily be checked that for every $z \in X_2 \setminus U$
the representation $\AAa_2 \ni a \mapsto [\Psi_2(a)](z) \in M_n$ vanishes on $\AAa_1$ (use
\eqref{eqn:homo} and the definition of $\Phi$). So, we conclude from the uniqueness of the extension
of the zero representation of $\AAa_1$ that $U = X_2$ and hence both $\varphi$ and $\Phi$ are
isomorphisms. Consequently, $\AAa_1 = \AAa_2$ and we are done.
\end{proof}

\SECT{Spectral theorem and $n$-functional calculus}

Whenever $\AAa$ is a unital $C^*$-algebra and $x_1,\ldots,x_k$ are arbitrary elements of $\AAa$,
let $C^*(x_1,\ldots,x_k)$ denote the $C^*$-subalgebra of $\AAa$ generated by $x_1,\ldots,x_k$
and let $C^*_1(x_1,\ldots,x_k)$ be the smallest $C^*$-subalgebra of $\AAa$ which contains
$x_1,\ldots,x_k$ as well as the unit of $\AAa$ (so, $C^*_1(x_1,\ldots,x_k) =
C^*(x_1,\ldots,x_k) + \CCC \cdot 1$ where $1$ is the unit of $\AAa$). We would like
to distinguish those systems $(x_1,\ldots,x_k)$ for which one of these two $C^*$-algebras defined
above is $n$-homogeneous. However, the property of being $n$-homogeneous is not hereditary for
$n > 1$. That is, when $n > 1$, \textit{every} nonzero $n$-homogeneous $C^*$-algebra contains
a $C^*$-subalgebra which is not $n$-homogeneous (namely, a nonzero commutative one). This causes
that the class of distinguished systems may depend on the choice of $C^*$-algebras related to them.
Fortunately, this does not happen, which is explained in the following

\begin{lem}{C*}
Let $\AAa$ be a unital $C^*$-algebra and $x_1,\ldots,x_k \in \AAa$. If $C^*_1(x_1,\ldots,x_k)$
is $n$-homogeneous for some $n > 1$, then $C^*_1(x_1,\ldots,x_k) = C^*(x_1,\ldots,x_k)$.
\end{lem}
\begin{proof}
Suppose, for the contrary, that the assertion is false. Observe that $\IIi :=
C^*(x_1,\ldots,x_k)$ is a two-sided ideal in $\BBb := C^*_1(x_1,\ldots,x_k)$, since $\BBb =
\IIi + \CCC \cdot 1$ ($1$ = the unit of $\AAa$). Moreover, $\BBb / \IIi$ is isomorphic
(as a $C^*$-algebra) to $\CCC$, which means that the canonical projection $\pi\dd \BBb \to \BBb /
\IIi$ may be considered as a $1$-dimensional (nonzero) representation. It is obviously irreducible,
which contradicts the fact that $\BBb$ is $n$-homogeneous (since $n > 1$).
\end{proof}

Taking into account the above result, we may now introduce

\begin{dfn}{sys-In}
A system $(x_1,\ldots,x_k)$ of elements of an (unnecessarily unital) $C^*$-algebra $\AAa$ is said
to be \textit{$n$-homogeneous} (where $n \geqsl 1$ is finite) if the $C^*$-subalgebra
$C^*(x_1,\ldots,x_k)$ of $\AAa$ generated by $x_1,\ldots,x_k$ is $n$-homogeneous.
\end{dfn}

This part of the paper is devoted to studies of (finite) $n$-homogeneous systems. We begin with

\begin{pro}{1-1}
Let $(x_1,\ldots,x_k)$ be an $n$-homogeneous system in a $C^*$-algebra $\AAa$. Let $(\Xx,.)$ be
the concrete $n$-spectrum of $C^*(x_1,\ldots,x_k)$ and let
\begin{equation}\label{eqn:spectrum}
\sigma_n(x_1,\ldots,x_k) := \{(\pi(x_1),\ldots,\pi(x_k))\dd\ \pi \in \Xx\}
\end{equation}
be equipped with the topology inherited from $(M_n)^k$ and with the action $$\uU.(A_1,\ldots,A_k) :=
(\uU.A_1,\ldots,\uU.A_k) \qquad (\uU \in \Uu_n,\ (A_1,\ldots,A_k) \in \sigma_n(x_1,\ldots,x_k)).$$
\begin{enumerate}[\upshape(Sp1)]
\item The pair $(\sigma_n(x_1,\ldots,x_k),.)$ is an $n$-space.
\item The function $H\dd (\Xx,.) \ni \pi \mapsto (\pi(x_1),\ldots,\pi(x_k)) \in
   (\sigma_n(x_1,\ldots,x_k),.)$ is an isomorphism of $n$-spaces.
\item Every member of $\sigma_n(x_1,\ldots,x_k)$ is irreducible; that is, if $(A_1,\ldots,A_k) \in
   \sigma_n(x_1,\ldots,x_k)$ and $T \in M_n$ commutes with each of $A_1,A_1^*,\ldots,A_k,A_k^*$,
   then $T$ is a scalar multiple of the unit matrix.
\item The set $\sigma_n(x_1,\ldots,x_k)$ is either compact or its closure in $(M_n)^k$ coincides
   with $\sigma_n(x_1,\ldots,x_k) \cup \{0\}$.
\end{enumerate}
\end{pro}
\begin{proof}
Let $\pi_0\dd \AAa \to M_n$ be the zero representation and let $\Omega = \Xx \cup \{\pi_0\}$ be
equipped with the pointwise convergence topology. Then $\Omega$ is compact (cf. the proof
of \THM{homog}). If $\pi_1, \pi_2 \in \Omega$, then the set $\{x \in C^*(x_1,\ldots,x_k)\dd\
\pi_1(x) = \pi_2(x)\}$ is a $C^*$-subalgebra of $C^*(x_1,\ldots,x_k)$. This implies that
the function $\widetilde{H}\dd \Omega \ni \pi \mapsto (\pi(x_1),\ldots,\pi(x_k)) \in
\sigma_n(x_1,\ldots,x_k) \cup \{0\}$ is one-to-one. It is obviously seen that $\widetilde{H}$ is
surjective and continuous. Consequently, $\widetilde{H}$ is a homeomorphism (since $\Omega$ is
compact). This proves (Sp4) and shows that $\sigma_n(x_1,\ldots,x_k)$ is locally compact. It is also
clear that $H(\uU.\pi) = \uU.H(\pi)$, which is followed by (Sp1) and (Sp2). Finally, for any $\pi
\in \Xx$, $C^*(\pi(x_1),\ldots,\pi(x_k)) = \pi(C^*(x_1,\ldots,x_k)) = M_n$ (see \LEM{irr}),
which yields (Sp3) and completes the proof.
\end{proof}

\begin{dfn}{n-spectrum-sys}
Let $(x_1,\ldots,x_k)$ be an $n$-homogeneous system in a $C^*$-algebra. The $n$-space
$(\sigma_n(x_1,\ldots,x_k),.)$ defined by \eqref{eqn:spectrum} is said to be
the \textit{$n$-spectrum} of $(x_1,\ldots,x_k)$. According to \PRO{1-1}, the $n$-spectrum
of $(x_1,\ldots,x_k)$ is an $n$-spectrum of $C^*(x_1,\ldots,x_k)$.
\end{dfn}

\begin{pro}{uniq}
Let $\textup{\xXx} = (x_1,\ldots,x_k)$ be an $n$-homogeneous system in a $C^*$-algebra. There
exists a unique $*$-homomorphism
$$\Phi_{\textup{\xXx}}\dd C^*(\sigma_n(\textup{\xXx}),.) \to C^*(\textup{\xXx})$$
such that $\Phi_{\textup{\xXx}}(p_j) = x_j$ where $p_j\dd \sigma_n(\textup{\xXx}) \ni
(A_1,\ldots,A_k) \mapsto A_j \in M_n\ (j=1,\ldots,k)$. Moreover, $\Phi_{\textup{\xXx}}$ is
a $*$-isomorphism of $C^*$-algebras.
\end{pro}
\begin{proof}
Let $(\Xx,.)$ be the concrete $n$-spectrum of $C^*(\textup{\xXx})$ and let $H\dd (\Xx,.) \to
(\sigma_n(\textup{\xXx}),.)$ be the isomorphism as in point (Sp2) of \PRO{1-1}. For $x \in
C^*(\textup{\xXx})$ let $\widehat{x} \in \CCc(\Xx,.)$ be given by $\widehat{x}(\pi) = \pi(x)$.
The proof of \THM{homog} shows that the function $C^*(\textup{\xXx}) \ni x \mapsto \widehat{x} \in
C^*(\Xx,.)$ is a $*$-isomorphism of $C^*$-algebras. Consequently, $\Psi\dd
C^*(\textup{\xXx}) \ni x \mapsto \widehat{x} \circ H^{-1} \in C^*(\sigma_n(\textup{\xXx}),.)$
is a $*$-isomorphism as well. A direct calculation shows that $\Psi(x_j) = p_j\ (j=1,\ldots,k)$.
This implies that $C^*(p_1,\ldots,p_k) = C^*(\sigma_n(\textup{\xXx}),.)$, from which we infer
the uniqueness of $\Phi_{\textup{\xXx}}$. To convince about its existence, just put
$\Phi_{\textup{\xXx}} = \Psi^{-1}$.
\end{proof}

We are now ready to introduce

\begin{dfn}{calc}
Let $\textup{\xXx} = (x_1,\ldots,x_k)$ be an $n$-homogeneous system and let $\Phi_{\textup{\xXx}}$
be as in \PRO{uniq}. For every $f \in C^*(\sigma_n(x_1,\ldots,x_k),.)$ we denote
by $f(x_1,\ldots,x_k)$ the element $\Phi_{\textup{\xXx}}(f)$. The assignment $f \mapsto
f(x_1,\ldots,x_k)$ is called the \textit{$n$-functional calculus}.
\end{dfn}

The reader familiar with functional calculus on normal operators (or normal elements
in $C^*$-algebras) has to be careful with the $n$-functional calculus, because its main
disadvantage is that its values are not $n$-homogeneous elements in general. Therefore we cannot
speak of the $n$-spectrum of $f(x_1,\ldots,x_k)$ in general. What is more, it may happen that
$\sigma_n(x_1,\ldots,x_k)$ is compact, but $j(x_1,\ldots,x_k)$, where $j$ is constantly equal
to the unit matrix, differs from the unit of the underlying $C^*$-algebra $\AAa$ from which
$x_1,\ldots,x_k$ were taken. This happens precisely when $C^*(x_1,\ldots,x_k)$ has unit, but this
unit is not the unit of $\AAa$.\par
As a consequence of \PRO{uniq} and \THM{represent} we obtain the \textit{spectral} theorem (for
$n$-homogeneous systems) announced before.

\begin{thm}{spectral}
Let $\tTT = (T_1,\ldots,T_k)$ be an $n$-homogeneous system of bounded linear operators acting
on a Hilbert space $\HHh$. There exists a unique spectral $n$-measure $E_{\tTT}\dd
\Bb(\sigma_n(\tTT)) \to M_n(\BBb(\HHh))$ such that
$$\int p_j \dint{E_{\tTT}} = T_j \qquad (j=1,\ldots,k)$$
where $p_j\dd \sigma_n(\tTT) \ni (A_1,\ldots,A_k) \mapsto A_j \in M_n$.
\end{thm}

\begin{dfn}{extend}
Let $\tTT = (T_1,\ldots,T_k)$ be an $n$-homogeneous system of bounded Hilbert space operators and
let $E_{\tTT}$ be the spectral $n$-measure as in \THM{spectral}. $E_{\tTT}$ is called
the \textit{spectral $n$-measure of $\tTT$} and the assignment $\Bb C^*(\sigma(\tTT),.) \ni f
\mapsto f(T_1,\ldots,T_n) := \int f \dint{E_{\tTT}} \in \BBb(\HHh)$ is called
the \textit{extended $n$-functional calculus}.
\end{dfn}

There is nothing surprising in the following

\begin{pro}{v-N}
Let $\MMm$ be a von Neumann algebra acting on a Hilbert space $\HHh$ and let $\tTT =
(T_1,\ldots,T_k)$ be an $n$-homogeneous system of operators belonging to $\MMm$. Let $X =
\sigma_n(\tTT)$.
\begin{enumerate}[\upshape(a)]
\item For any $f \in \Bb C^*(X,.)$, $f(\tTT) \in \MMm$.
\item If $f^{(1)},f^{(2)},\ldots \in \Bb C^*(X,.)$ converge pointwise to $f\dd X \to M_n$ and
   $$\sup_{\substack{m\geqsl1\\x \in X}} \|f^{(m)}(x)\| < \infty,$$
   then $f \in \Bb C^*(X,.)$ and $\lim_{m\to\infty} (f^{(m)}(\tTT)) h = (f(\tTT)) h$ for each
   $h \in \HHh$.
\end{enumerate}
\end{pro}
\begin{proof}
We begin with (a). It is clear that $g(\tTT) \in \MMm$ for $g \in C^*(X,.)$. Let $E_{\tTT} =
[E_{pq}]$. Denote by $\scalarr$ the scalar product of $\HHh$ and fix $f = [f_{pq}] \in
\Bb C^*(X,.)$. We shall show that $f(\tTT)$ belongs to the closure of $\{g(\tTT)\dd\ g \in
C^*(X,.)\}$ in the weak operator topology of $\BBb(\HHh)$, which will give (a). To this end,
we fix $h_1,w_1,\ldots,h_r,w_r \in \HHh$ and $\epsi > 0$. Put $\mu = \sum_{s=1}^r \sum_{p,q}
|E_{pq}^{(h_s,w_s)}|$. By \LEM{dense}, there is $g = [g_{pq}] \in C^*(X,.)$ such that $\int_X
\|f(x) - g(x)\| \dint{\mu(x)} \leqsl \epsi$. But then, for each $s \in \{1,\ldots,r\}$,
\begin{multline*}
\left|\Bigl\langle\Bigl(\int f \dint{E_{\tTT}} - \int g \dint{E_{\tTT}}\Bigr)
h_s,w_s\Bigr\rangle\right| = \Bigl|\sum_{p,q} \int_X (f_{pq} - g_{pq})
\dint{E_{qp}^{(h_s,w_s)}}\Bigr|\\\leqsl \sum_{p,q} \int_X |f_{pq} - g_{pq}|
\dint{|E_{qp}^{(h_s,w_s)}|} \leqsl \int_X \|f(x) - g(x)\| \dint{\mu(x)} \leqsl \epsi
\end{multline*}
and we are done (since $f(\tTT) = \int f \dint{E_{\tTT}}$ and $g(\tTT) =
\int g \dint{E_{\tTT}}$).\par
We turn to (b). It is clear that $f \in \Bb C^*(X,.)$. Replacing $f^{(m)}$ by $f^{(m)} - f$,
we may assume $f = 0$. Observe that then $\lim_{m\to\infty} ((f^{(m)})^* f^{(m)})_{pq}(x) = 0$ for
any $x \in X$ and $p, q \in \{1,\ldots,n\}$ and the functions $((f^{(1)})^* f^{(1)})_{pq},
((f^{(2)})^* f^{(2)})_{pq},\ldots$ are uniformly bounded. Therefore (by Lebesgue's dominated
convergence theorem) for any $h \in H$,
\begin{multline*}
\|(f^{(m)}(\tTT)) h\|^2 = \scalar{(f^{(m)}(\tTT))^* (f^{(m)}(\tTT)) h}{h} =
\Bigl\langle\Bigl(\int (f^{(m)})^* f^{(m)} \dint{E_{\tTT}}\Bigr)h,h\Bigr\rangle\\
= \sum_{p,q} \int_X ((f^{(m)})^* f^{(m)})_{pq} \dint{E_{qp}^{(h,h)}} \to 0 \quad (m\to\infty),
\end{multline*}
which finishes the proof.
\end{proof}

We end the paper with the note that the above result enables defining the extended $n$-functional
calculus for $n$-homogeneous systems in $W^*$-algebras.

\end{document}